\documentclass{amsart}

\usepackage{mathptmx}

\usepackage{a4wide}
\usepackage{amsmath,amssymb}
\usepackage{graphicx}
\usepackage[all]{xy}

\usepackage{color}
\definecolor{Bordeaux}{rgb}{0.545, 0.137, 0.137}
\definecolor{BleuGris}{rgb}{0.212, 0.392, 0.545}
\definecolor{Chocolat}{rgb}{0.36, 0.2, 0.09}
\definecolor{BleuTresFonce}{rgb}{0.215, 0.215, 0.36}

\usepackage[colorlinks,final,hyperindex]{hyperref}
\hypersetup{citecolor=BleuTresFonce, linkcolor=Chocolat, urlcolor=BleuTresFonce}

\DeclareMathAlphabet{\mathbbold}{U}{bbold}{m}{n}

\def\k{\mathbbold{k}}

\DeclareMathOperator{\id}{id}
\DeclareMathOperator{\ad}{ad}
\DeclareMathOperator{\sgn}{sgn}
\DeclareMathOperator{\Lie}{Lie}

\DeclareMathOperator{\Hom}{Hom}
\DeclareMathOperator{\End}{End}
\DeclareMathOperator{\MC}{MC}
\DeclareMathOperator{\Sym}{Sym}

\DeclareMathOperator{\Com}{Com}

\newcommand{\ac}{{\ensuremath{\scriptstyle \text{\rm !`}}}}

\newtheorem {theorem}{Theorem}

\newtheorem {corollary}{Corollary}

\newtheorem {proposition}{Proposition}

\theoremstyle{definition}

\newtheorem {definition}{Definition}
\newtheorem {remark}{Remark}

\newcommand{\frg}{\mathfrak{g}}

\begin{document}

\title{A tale of three homotopies}

\author{Vladimir Dotsenko \and Norbert Poncin}

\address{School of Mathematics,  Trinity College, Dublin 2, Ireland, e-mail: \texttt{vdots@maths.tcd.ie}}

\address{Mathematics Research Unit,
	 University of Luxembourg,
	 Campus Kirchberg,
	 6, Rue Richard Coudenhove-Kalergi,
	 L-1359 Luxembourg,
	 Grand Duchy of Luxembourg,
e-mail: \texttt{norbert.poncin@uni.lu}}

\thanks{The research of the authors was supported by Grant	GeoAlgPhys 2011-2013 awarded by the University of Luxembourg.}


\maketitle

\begin{abstract}
For a Koszul operad $\mathcal{P}$, there are several existing approaches to the notion of a homotopy between homotopy morphisms of homotopy $\mathcal{P}$-algebras. Some of those approaches are known to give rise to the same notions. We exhibit the missing links between those notions, thus putting them all into the same framework. The main nontrivial ingredient in establishing this relationship is the homotopy transfer theorem for homotopy cooperads due to Drummond-Cole and Vallette.
\end{abstract}


\section{Introduction}

From as early as Quillen's work on rational homotopy theory \cite{Quillen69}, equivalences of various homotopy categories of algebras have proved to be one of the key tools of homotopical algebra. (This paper does not aim to serve as a historical reference, so we refer the reader to \cite{Huebschmann07} and references therein). The types of algebras for which the corresponding homotopy categories have attracted most attention over years are, eloquently described by Jean-Louis Loday, ``the three graces'', that is associative algebras, associative commutative algebras, and Lie algebras. However, the corresponding questions make sense for any type of algebras, or, in a more modern language, for algebras over any operad. For instance, for the algebra of dual numbers $\k[\epsilon]/(\epsilon^2)$ viewed as an operad with unary operations only, algebras are chain complexes, and a good understanding of the corresponding homotopy category naturally leads to the notion of a spectral sequence~\cite{Lapin01}. In general, a 
``nice'' homotopy theory of algebras over an operad $\mathcal{P}$ is available in the case of any Koszul operad. More precisely, there are several equivalent ways to relax a notion of a dg (standing for differential graded) $\mathcal{P}$-algebra up to homotopy, and define appropriate homotopy morphisms of homotopy algebras.

\smallskip

Although a few available ways to write down a definition of a homotopy $\mathcal{P}$-algebras and a homotopy morphism between two homotopy algebras are easily seen to be equivalent to one another, in order to describe the homotopy category of dg $\mathcal{P}$-algebras one has also to be able to encode homotopy relations between homotopy morphisms. (Another instance where this question naturally is raised comes from the informal relationship between the categorification of $\mathcal{P}$-algebras and relaxing $\mathcal{P}$-algebras up to homotopy, see, e.~g.~\cite{BC04,KMP11}). Basically, there are at least the following three natural candidates to encode homotopies between morphisms:
\begin{itemize}
 \item The \emph{concordance} relation between homotopies, based on two different augmentations of the dg algebra $\Omega([0,1])$ of differential forms on the interval (this notion is discussed in \cite{SSS} in detail; it seems to have first appeared in unpublished work of Stasheff and Schlessinger \cite{SS84} and is inspired by a paper of Bousfield and Gugenheim \cite{BG76})
 \item Several notions of homotopy relations based on the interpretation of homotopy morphisms as Maurer--Cartan elements in a certain $L_\infty$-algebra:
    \begin{itemize}
      \item The \emph{Quillen homotopy} notion (close to the above notion of concordance) suggesting that two Maurer--Cartan elements in an algebra $L$ are homotopic if they are images of the same Maurer--Cartan element in $L[t,dt]$ under two different morphisms to~$L$
      \item The \emph{gauge homotopy} notion suggesting that the component $L_0$ of an $L_\infty$-algebra $L$ acts on Maurer--Cartan elements, and homotopy classes are precisely orbits of that action. Gauge symmetries of Maurer--Cartan elements in differential graded Lie algebras are already somewhat prominent in the seminal paper of Nijenhuis and Richardson~\cite{NR2}; their role has been further highlighted by Schlessinger and Stasheff~\cite{SS84}, and promoted to the context of 2-groupoids by Deligne~\cite{Del}. A systematic treatment of gauge symmetries of Maurer--Cartan elements in $L_\infty$-algebras is due to Getzler~\cite{Get}, and his methods were specifically used to define homotopy of $L_\infty$-morphisms by Dolgushev~\cite{Dol07}.
      \item The \emph{cylinder homotopy} notion coming from the cylinder construction of the dg Lie algebra controlling Maurer--Cartan elements; such a cylinder is shown \cite{BM2} to be given by the Lawrence--Sullivan construction~\cite{LS}.
    \end{itemize}
 \item The notion of \emph{operadic homotopy} suggesting that the datum of two homotopy algebras, two homotopy morphisms between them, and a homotopy between those two morphisms is the same as the datum of an algebra over a certain cofibrant replacement of the coloured operad describing the diagram
 \[
\xymatrix@C=3em @!C{
X
        \ar@/^3ex/ [r] ^-{p} ^-{}="src"
        \ar@/_3ex/ [r] _-{q} _-{}="tgt"
& Y
\ar@{=} "src"!<0pt,-10pt>;"tgt"!<0pt,10pt>
}
 \]
of $\mathcal{P}$-algebras (this approach \emph{\`a la} Boardman and Vogt~\cite{BV73} was pursued by Markl in~\cite{Mar2}, following the description of homotopy algebras and homotopy morphisms via algebras over minimal models of appropriate operads \cite{Mar1}).
\end{itemize}

\smallskip

The goal of this paper is to exhibit, for a Koszul operad $\mathcal{P}$, interrelationships between these definitions, putting the above approaches in a common context. For some notions of homotopies between Maurer--Cartan elements, it is done in a recent preprint \cite{BM1}. The interplay between concordance, Quillen homotopy, and operadic homotopy is explained in this paper. This would be useful for working with homotopy category of homotopy $\mathcal{P}$-algebras, as in~\cite{Val12}. 

A very important computation which is in a way at heart of both some very interesting recent results in rational homotopy theory \cite{BM1,BM2} and our theorem on operadic homotopy is homotopy transfer of the dg commutative algebra structure of differential forms on the interval leading to a homotopy commutative algebra structure on the \v{C}ech cochain complex of the interval with Bernoulli numbers as structure constants. This computation was first performed by Cheng and Getzler~\cite{CG}. In \cite{BM1}, a version of the computation of Cheng and Getzler was performed on the dual level, resulting on a homotopy cocommutative coalgebra structure on the \v{C}ech chain complex of the interval which they show to recover the universal enveloping algebra of the Lawrence--Sullivan Lie algebra. There is a subtle point in this statement: the dual of the algebra of differential forms is not a coalgebra, since the coproduct lands in the \emph{completed} tensor product, however if one ignores the fact that intermediate 
computations involve infinite series that technically do not exist, the transferred structure is an honest homotopy coalgebra. In our case, since we perform a similar computation but transfer the structure of a homotopy cooperad (using results of Drummond-Cole and Vallette~\cite{DV}), the situation becomes even more subtle. We therefore create a framework that justifies the infinite series computations, proving directly that partial sums of those infinite series give higher structures that converge as the upper summation limit goes to infinity.

\smallskip

The paper is organised as follows. In Section~\ref{background-operads}, we briefly recall all necessary definitions and facts of operadic homotopical algebra. In Section~\ref{homotopies}, we provide background information on the existing notions of homotopies; even though the three different notions of a homotopy between Maurer--Cartan elements in $L_\infty$-algebras are fairly well understood, we spell out the corresponding definitions for the sake of completeness. In Section~\ref{MC-and-Quillen}, we explain the relationship between the notion of concordance homotopy and that coming from homotopy of Maurer--Cartan elements. In Section~\ref{concordance-and-Markl}, we explain the relationship between the notion of concordance homotopy and that of operadic homotopy. In fact we provide the first, to our knowledge, explicit recipe to write a definition of operadic homotopy, even though it is complicated since it involves nested trees in homotopy transfer formulae. We conclude with an outline of some future directions in Section~\ref{future}.

\subsection*{Acknowledgements} 
We would like to thank David Khudaverdyan for inspiring discussions, Urtzi Buijs and Aniceto Murillo for sending copies of their related work, and Alberto Canonaco for sending a copy of~\cite{Can1}. Special thanks are due to  Bruno Vallette for numerous useful discussions and for sharing a preliminary version of \cite{Val12}, to Ezra Getzler for enlightening discussions of  \cite{Get}, and to Martin Markl and Martin Doubek for discussions of rich ideas of \cite{Mar2} and \cite{Doubek}.  
Some extensive work on this paper was done while the first author was visiting Maria Ronco at University of Talca; he is most grateful for the invitation and the excellent working conditions enjoyed during his stay there.

\section{Operadic homotopical algebra}\label{background-operads}

We do not aim to provide a comprehensive treatment of homotopical algebra for operads since it would require a textbook rather than a paper; we however tried to collect all basic notation, slightly uncommon definitions and some proofs of facts we could not locate in the available literature. We  refer the reader to \cite{LV} for all the missing details.

\subsection{Operads: notational (and other) conventions}

All vector spaces are defined over a field $\k$ of characteristic~$0$. We shall use coloured operads throughout, and therefore we find it beneficial to recall some definitions, directing the reader to \cite{LV} for definitions of the corresponding non-coloured notions. More details on coloured operads can be found in \cite{vdL}. For a set $\mathsf{C}$, a \emph{$\mathsf{C}$-coloured $\mathbb{S}$-module} is a functor from the category of $\mathsf{C}$-coloured finite sets (with colour-preserving bijections as morphisms) to $\mathsf{C}$-graded vector spaces. Similarly to how a non-coloured $\mathbb{S}$-module $\mathcal{V}$ is completely determined by the components $\mathcal{V}(n):=\mathcal{V}(\{1,2,\ldots,n\})$, a $\mathsf{C}$-coloured $\mathbb{S}$-module $\mathcal{V}$ is completely determined by its \emph{components} $\mathcal{V}(c_1,\ldots,c_n):=\mathcal{V}(\{(1,c_1),(2,c_2),\ldots,(n,c_n)\})$ for $c_1,\ldots,c_n\in\mathsf{C}$.

 In some instances, we shall use \emph{$\mathsf{C}$-graded chain complexes}, that is, $\mathsf{C}$-graded vector spaces for which each individual component is a chain complex.  The category of $\mathsf{C}$-coloured $\mathbb{S}$-modules has an important object that we denote by $\mathcal{I}$; it is the functor that vanishes on all sets except one-element sets, and on a one-element set with the only element of colour $c$, its value is the $\mathsf{C}$-graded vector space whose only nonzero component is that of colour $c$, and that component is one-dimensional. For a non-coloured chain complex $U$, and $c\in\mathsf{C}$, we denote by $U_c$ the $\mathsf{C}$-graded chain complex whose only nonzero component is the $c$-graded one, and it is equal to $C$. For a non-coloured $\mathbb{S}$-module $\mathcal{V}$, and $c_1,c_2\in\mathsf{C}$, we denote by $\mathcal{V}_{(c_1\to c_2)}$ the $\mathsf{C}$-coloured $\mathbb{S}$-module whose only nonzero components are $\mathcal{V}_{(c_1\to c_2)}(\underbrace{c_1,c_1,\ldots,c_1}_{n \text{ times}}):=\mathcal{V}(n)_{c_2}$.

The category of $\mathsf{C}$-coloured $\mathbb{S}$-modules has a well known monoidal structure called \emph{composition} and denoted by $\circ$ for which $\mathcal{I}$ is the unit; monoids in this category are called \emph{$\mathsf{C}$-coloured operads}. For a $\mathsf{C}$-graded vector space $Z$, the \emph{coloured endomorphism operad} $\End_Z$ is the $\mathsf{C}$-coloured operad whose component $\End_Z(c_1,\ldots,c_n)$ is the $\mathsf{C}$-graded vector space with the $c$-graded component being $\Hom(Z_{c_1}\otimes\cdots\otimes Z_{c_n},Z_c)$, and obvious composition maps. An \emph{algebra} over a $\mathsf{C}$-coloured operad $\mathcal{O}$ is a $\mathsf{C}$-graded vector space $Z$ together with a morphism of coloured operads $\mathcal{O}\to\End_Z$. The additional characteristics ``coloured'' and ``differential graded'' that an operad or an $\mathbb{S}$-module may have will always be clear from the context, and we shall use just the words ``operad'' and ``$\mathbb{S}$-module'' in most cases for brevity. It is also worth recalling that besides the composition $\mathcal{V}\circ\mathcal{W}$, one can also define the \emph{infinitesimal composition} $\mathcal{V}\circ_{(1)}\mathcal{W}$, which consists of the elements of $\mathcal{V}\circ(\mathcal{I}\oplus\mathcal{W})$ that are linear in $\mathcal{W}$. 

To handle suspensions, we introduce a formal symbol~$s$ of degree~$1$. For a graded vector space~$L$, its suspension $sL$ is nothing but $\k s\otimes L$. For an augmented (co)operad $\mathcal{O}$ (for example, for every (co)operad $\mathcal{O}$ with $\dim\mathcal{O}(1)=1$), we denote by $\overline{\mathcal{O}}$ its augmentation (co)ideal. 

We shall frequently use the chain complex $C_\bullet([0,1])$,  the \v{C}ech chain complex of the interval. It is the chain complex that has basis elements $\mathbf{0}$, $\mathbf{1}$, and $\mathbf{01}$ of degrees $0$, $0$, and $1$ respectively, and the differential $\partial(\mathbf{01})=\mathbf{1}-\mathbf{0}$. 

\subsection{Operadic Koszul duality and homotopy algebras}\label{sec:operadic-recollection}

Given an  $\mathbb{S}$-module $\mathcal{V}$, one can define the \emph{free operad} $\mathcal{F}(\mathcal{V})$ generated by $\mathcal{V}$ and the \emph{cofree cooperad} $\mathcal{F}^c(\mathcal{V})$ generated by $\mathcal{V}$; as $\mathbb{S}$-modules, they both are spanned by ``tree-shaped tensors''. Each of them admits a weight grading, e.g. $\mathcal{F}(\mathcal{V})^{(k)}$ is spanned by tree-shaped tensors corresponding to trees with $k$ vertices, or, in other words, by composites of $k$ generators. 

A dg operad is called \emph{quasi-free} if its underlying operad is free. A \emph{model} of an non-dg operad $\mathcal{O}$ is a quasi-free operad $(\mathcal{F}(\mathcal{U}),d)$ equipped with a surjective quasi-isomorphism $(\mathcal{F}(\mathcal{U}),d)\to \mathcal{O}$. We shall use the definition of minimal models for operads from \cite{DV} which is more general than the one from \cite{Mar3}, and is required for our purposes. Namely, we say that a quasi-free operad $(\mathcal{F}(\mathcal{U}),d)$ is \emph{minimal} if its differential is decomposable, that is $d(\mathcal{U})\subset \mathcal{F}(\mathcal{U})^{(\ge 2)}$, and its $\mathbb{S}$-module of generators admits a direct sum decomposition $\mathcal{U}=\bigoplus_{k\ge 1}\mathcal{U}^{(k)}$ satisfying $d(\mathcal{U}^{(k+1)})\subset \mathcal{F}(\bigoplus_{i=1}^k\mathcal{U}^{(i)})$, the \emph{Sullivan triangulation condition}.

To an $\mathbb{S}$-module $\mathcal{V}$ and an $\mathbb{S}$-submodule $\mathcal{R}\subset\mathcal{F}(\mathcal{V})^{(2)}$ one can associate an operad $$\mathcal{P}=\mathcal{P}(\mathcal{V},\mathcal{R}),$$ the universal quotient operad $\mathcal{O}$ of $\mathcal{F}(\mathcal{V})$ for which the composite
\[
\mathcal{R}\hookrightarrow\mathcal{F}(\mathcal{V})\twoheadrightarrow\mathcal{O}
\]
is zero. Similarly, to an $\mathbb{S}$-module $\mathcal{V}$ and an $\mathbb{S}$-submodule $\mathcal{R}\subset\mathcal{F}^c(\mathcal{V})^{(2)}$ one can associate a cooperad $\mathcal{Q}=\mathcal{Q}(\mathcal{V},\mathcal{R})$, the universal suboperad $\mathcal{C}\subset\mathcal{F}^c(\mathcal{V})$ for which the composite
\[
\mathcal{C}\hookrightarrow\mathcal{F}^c(\mathcal{V})\twoheadrightarrow\mathcal{F}^c(\mathcal{V})^{(2)}/\mathcal{R}
\]
is zero. The Koszul duality for operads assigns to an operad $\mathcal{P}=\mathcal{P}(\mathcal{V},\mathcal{R})$ its Koszul dual cooperad $$\mathcal{P}^\ac:=\mathcal{Q}(s\mathcal{V},s^2\mathcal{R}),$$ and to a cooperad $\mathcal{Q}=\mathcal{Q}(\mathcal{V},\mathcal{R})$ its Koszul dual operad $$\mathcal{Q}^\ac:=\mathcal{P}(s^{-1}\mathcal{V},s^{-2}\mathcal{R}).$$ An operad $\mathcal{P}$ is said to be \emph{Koszul} if its Koszul complex  ($\mathcal{P}^{\ac}\circ\mathcal{P}$ with the differential coming from a certain twisting morphism between $\mathcal{P}^{\ac}$ and $\mathcal{P}$) is quasi-isomorphic to $\mathcal{I}$. 

It is well known that if $\mathcal{P}$ is a Koszul operad, then the datum of a homotopy $\mathcal{P}$-algebra structure on a vector space $V$ is equivalent to the datum of a square zero coderivation of degree~$-1$ of the cofree $\mathcal{P}^\ac$-coalgebra $\mathcal{P}^\ac(V)$. Such a coderivation makes the latter coalgebra into a chain complex referred to as the bar complex of~$V$, and denoted $\mathsf{B}(V)$. For every homotopy $\mathcal{P}$-algebra structure on $V$,  we shall denote by $D_V$ the differential of  $\mathsf{B}(V)$, and by $d^{(k)}_V$ the $k$-th restriction of $D_V$, which is a composite of the restriction of $D_V$ to $\mathcal{P}^{\ac}(k)\otimes_{S_k}V^{\otimes k}\subset \mathcal{P}^{\ac}(V)$ and the projection $\mathcal{P}^\ac(V)\twoheadrightarrow V$. 

The same definitions apply when replacing algebras with coalgebras: for a Koszul cooperad $\mathcal{Q}$, a structure of a homotopy $\mathcal{Q}$-coalgebra on a vector space $V$ is exactly the same as a square zero derivation of degree~$1$ of the free $\mathcal{Q}^\ac$-algebra $\mathcal{Q}^\ac(V)$. Such a datum makes the latter coalgebra into a cochain complex referred to as the cobar complex of~$V$ and denoted $\Omega(V)$. 

The above statements also apply to the case when $V$ itself is a homotopy (co)operad, that is a homotopy (co)algebra over the (Koszul) coloured (co)operad encoding non-coloured operads. In the case of non-symmetric operads, that coloured operad is defined and studied in detail in \cite{vdL}, in the case of symmetric operads, the definition is given in \cite{BM07,KS}. We however would like to make some clarifying remarks since when applying the Koszul duality to that operad one may make different choices, and end up with several different notions of homotopy operads, see, e.g. a recent preprint \cite{DV15} where the action of symmetric groups on operads is also relaxed up to homotopy. 
We consider operads coloured by a category~\cite{Pet12}: the set of colours for our operads is $\mathbb{N}_+=\{1,2,\ldots\}$, but in addition each colour $n$ has $S_n$ as its group of automorphisms. Hence, the coloured collections underlying the corresponding coloured operads will be collections of vector spaces $V(c_1,\ldots,c_n;c)$ that, in addition to the action of permutations corresponding to same colours in the list $c_1,\ldots,c_n$, have a left $\k S_{c_1}\otimes\cdots\otimes\k S_{c_n}$-module structure, and a right $\k S_c$-module structure, and the left $S_{c_1}\otimes\cdots\otimes\k S_{c_n}$-module structure is compatible with permutations of colours. For such operads, it is possible to generalise the notion of Koszul duality \emph{\`a la} \cite{GK}, the notion of a Gr\"obner basis \emph{\`a la} \cite{DK10}, and various results of operadic homotopical algebra \emph{\`a la} \cite[Chapter 10]{LV}; these generalisations, while would require a separate paper to fill in all the details, are fairly straightforward, and we are using them implicitly in several proofs throughout this paper. In fact, the only coloured operad of this generalised form that we need is the coloured operad $\mathcal{O}$ with generators $\alpha_{i,\sigma}\in\mathcal{O}(n,m;n+m-1)$, $1\le i\le n$, $\sigma\in S_{m+n-1}$; this operation encodes infinitesimal operadic compositions via $\alpha_{i,\sigma}(f,g)=(f\circ_i g).\sigma$. This operad is presented by quadratic relations that encode associativity of operadic compositions~\cite{Mar3}. Moreover, one of several standard choices of normal forms for computing operadic compositions (e.g. by choosing left-to-right levelisations of trees) leads to a conclusion that this operad  is Koszul because it can be easily seen to admit a quadratic Gr\"obner basis~\cite{DK10}, and by a direct inspection, this operad is self-dual with respect to Koszul duality for coloured operads. 
Even more generally, for a given set of colours $\mathsf{C}$ (without nontrivial automorphisms), a $\mathsf{C}$-coloured homotopy (co)operad $V$ can be viewed as a homotopy (co)algebra over an appropriate coloured (co)operad $\mathcal{O}_{\mathsf{C}}$; this coloured (co)operad satisfies all the properties we just outlined for $\mathsf{C}=\{*\}$. 

In particular, this translates into the fact that for an $\mathbb{S}$-module $\mathcal{W}$, a square zero derivation of degree $-1$ of the free operad $\mathcal{F}(s^{-1}\mathcal{W})$ is equivalent to a structure of a homotopy cooperad on $\mathcal{W}$, see~\cite{vdL}; in fact, for a cooperad $\mathbb{S}$-module $\mathcal{W}$, the free operad $\mathcal{F}(s^{-1}\mathcal{W})$ equipped with that differential is precisely the cobar complex $\Omega(\mathcal{W})$. In terms of operadic cobar complexes, one can give an alternative definition of homotopy algebras over Koszul operads: a homotopy $\mathcal{P}$-algebra structure (or a $\mathcal{P}_\infty$-algebra structure) on a chain complex $V$ is the same as the structure of an algebra over the cobar complex $\Omega(\overline{\mathcal{P}^\ac})$ on $V$. (This cobar complex is often denoted by $\mathcal{P}_\infty$). Similarly to how cooperations of an $A_\infty$-coalgebra are indexed by positive integers (the label of a cooperation describes in how many parts it splits its argument), cooperations of a homotopy cooperad $\mathcal{W}$ are indexed by trees. For each tree $t$, the cooperation $\Delta_t\colon s^{-1}\mathcal{W}\to \mathcal{F}(s^{-1}\mathcal{W})$ takes an element of $s^{-1}\mathcal{W}$ to a sum of terms in the free operad $\mathcal{F}(s^{-1}\mathcal{W})$, each term corresponding to a certain way to decorate internal vertices of $t$ by elements of $s^{-1}\mathcal{W}$. It is of course possible to encode these as maps from $\mathcal{W}$ to $\mathcal{F}(\mathcal{W})$ by applying appropriate (de)suspensions, thus arriving to the more conventional definition where the infinitesimal decomposition map in a usual cooperad has degree $0$. 

Throughout this paper, we always use the letter $\mathcal{P}$ to denote a non-coloured non-graded finitely generated Koszul operad with $\mathcal{P}(1)=\k$.  The ``non-coloured'' assumption is merely there to simplify the notation a little bit (all the results hold in the coloured case also), while the other assumptions cannot be just dropped, while each of them can in principle be replaced by a more weak but more technical assumption, e.g. instead of considering operads with $\mathcal{P}(1)=\k$ one can look at augmented operads admitting a minimal model in the sense described above. 
Under our assumptions, the cobar complex $\Omega(\overline{\mathcal{P}^{\ac}})$ is the minimal model of $\mathcal{P}$. 

\subsection{Morphisms and homotopy morphisms}\label{infinitymor}

To deal with homotopy algebras and their morphisms, we shall mainly use $\{x,y\}$-coloured $\mathbb{S}$-modules and operads.  For a non-coloured operad $\mathcal{P}=\mathcal{F}(\mathcal{V})/(\mathcal{R})$, a pair of $\mathcal{P}$-algebras and an algebra morphism between them can be encoded as an algebra over a certain $\{x,y\}$-coloured operad $\mathcal{P}_{\bullet\to\bullet}$. The generators of that operad are $\mathcal{V}_{(x\to x)}$ (encoding the structure maps of the first algebra), $\mathcal{V}_{(y\to y)}$ (encoding the structure maps of the second algebra), and the $\mathbb{S}$-module $\mathcal{M}_{(f)}$, for which the only nonzero component is $\mathcal{M}_{(f)}(x)=\k_y$ (encoding the map between the two algebras). Its relations are $\mathcal{R}_{(x\to x)}$, $\mathcal{R}_{(y\to y)}$, and $f\circ v_{(x\to x)}-v_{(y\to y)}\circ f^{\otimes n}$ for each $v\in\mathcal{V}(n)$. This operad is homotopy Koszul in the sense of \cite{MV}; we shall recall its minimal model below.

Recall that a homotopy morphism between two homotopy $\mathcal{P}$-algebras is the same as a dg $\mathcal{P}^\ac$-coalgebra morphism between their bar complexes. (Dually, a homotopy morphism between two homotopy $\mathcal{Q}$-coalgebras is the same as a dg $\mathcal{Q}^\ac$-algebra morphism between their cobar complexes).
Similarly to how a homotopy $\mathcal{P}$-algebra structure can be defined as an algebra over the operad $\mathcal{P}_\infty=\Omega(\overline{\mathcal{P}^\ac})$, there is a description of homotopy morphisms in terms of algebras over some dg operad, which we shall now define. 

Let us consider the $\{x,y\}$-coloured $\mathbb{S}$-module
 \[
\mathcal{V}_{\bullet\to\bullet,\infty}:=\overline{\mathcal{P}^{\ac}}_{(x\to x)}\oplus \overline{\mathcal{P}^{\ac}}_{(y\to y)}\oplus s\mathcal{P}^{\ac}_{(x\to y)}  ,
 \]
It has a structure of a homotopy cooperad defined as follows: on $\overline{\mathcal{P}^{\ac}}_{(x\to x)}$ and $\overline{\mathcal{P}^{\ac}}_{(y\to y)}$, one uses the cooperad structure corresponding to that of $\overline{\mathcal{P}^{\ac}}$, whereas on $s\mathcal{P}^{\ac}_{(x\to y)}$ there are two types of nonzero decomposition maps, the map
 \[
\mathcal{P}^{\ac}_{(x\to y)}\cong \mathcal{P}^{\ac}\to \mathcal{P}^{\ac}\circ_{(1)}\overline{\mathcal{P}^{\ac}}\cong\mathcal{P}^{\ac}_{(x\to y)}\circ_{(1)}s^{-1}\overline{\mathcal{P}^{\ac}}_{(x\to x)}  
 \]
obtained by de-suspending the infinitesimal decomposition $\mathcal{P}^{\ac}\to \mathcal{P}^{\ac}\circ_{(1)}\mathcal{P}^{\ac}$ in the cooperad $\mathcal{P}^{\ac}$, and the map
 \[
 \mathcal{P}^{\ac}_{(x\to y)}\cong \mathcal{P}^{\ac}\to \overline{\mathcal{P}^{\ac}}\circ \mathcal{P}^{\ac}\cong \overline{s^{-1}\mathcal{P}^{\ac}}_{(y\to y)}\circ \mathcal{P}^{\ac}_{(x\to y)}  
 \]
obtained by de-suspending the full decomposition $\mathcal{P}^{\ac}\to \mathcal{P}^{\ac}\circ\mathcal{P}^{\ac}$. The fact that all these maps satisfy the constraints required by the definition of a homotopy cooperad follow from the fact that the structure  maps of $\mathcal{P}^{\ac}$ satisfy the constraints of a cooperad (coassociativity). 

The following proposition follows by inspection from the definition of a homotopy morphism as the morphism of bar complexes; we omit the proof. 

\begin{proposition}
The datum of a homotopy morphism between two homotopy $\mathcal{P}$-algebras $X$ and $Y$ is equivalent to an algebra over the cobar complex $\Omega(\mathcal{V}_{\bullet\to\bullet,\infty})$ for which the actions of  $\overline{\mathcal{P}^{\ac}}_{(x\to x)}$ and $\overline{\mathcal{P}^{\ac}}_{(y\to y)}$ induce the given homotopy $\mathcal{P}$-algebra structures on $X$ and $Y$ respectively.	
\end{proposition}

In what follows, we shall denote the cobar complex $\Omega(\mathcal{V}_{\bullet\to\bullet,\infty})$ by  $\mathcal{P}_{\bullet\to\bullet,\infty}$. The following statement extends the understanding of homotopy $\mathcal{P}$-algebras as algebras over the minimal model of $\mathcal{P}$; it is essentially \cite[Prop.~56]{MV} for which we provide a detailed proof.

\begin{proposition}
The operad $\mathcal{P}_{\bullet\to\bullet,\infty}$ is the minimal model of $\mathcal{P}_{\bullet\to\bullet}$. 
\end{proposition}

\begin{proof}
Let us consider the weight filtration of the cobar complex $\Omega(\mathcal{V}_{\bullet\to\bullet,\infty})$, that is the filtration by the number of internal vertices of trees in the free operad. Inspecting the definition of the homotopy cooperad $\mathcal{V}_{\bullet\to\bullet}$, we see that the differential $d^0$ of the corresponding spectral sequence is equal to zero, and the differential $d^1$ is obtained from forgetting the full decomposition map on $s\mathcal{P}^{\ac}_{(x\to y)}$, that is only retaining the map
\[
\mathcal{P}^{\ac}_{(x\to y)}\cong \mathcal{P}^{\ac}\to \mathcal{P}^{\ac}\circ_{(1)}\overline{\mathcal{P}^{\ac}}\cong\mathcal{P}^{\ac}_{(x\to y)}\circ_{(1)}s^{-1}\overline{\mathcal{P}^{\ac}}_{(x\to x)}  
\]
obtained by de-suspending the infinitesimal decomposition $\mathcal{P}^{\ac}\to \mathcal{P}^{\ac}\circ_{(1)}\mathcal{P}^{\ac}$ in the cooperad $\mathcal{P}^{\ac}$. Thus, the cobar complex $\Omega(\mathcal{V}_{\bullet\to\bullet,\infty})$ with the differential $d^1$ becomes isomorphic to $\Omega((\mathcal{P}^{(2)}_{\bullet\to\bullet})^{\ac})$, the cobar complex of the operad $\mathcal{P}^{(2)}_{\bullet\to\bullet}$ with generators $\mathcal{V}_{(x\to x)}\oplus\mathcal{V}_{(y\to y)}\oplus\mathcal{M}_{(f)}$ and relations $\mathcal{R}_{(x\to x)}$, $\mathcal{R}_{(y\to y)}$, and $f\circ v_{(x\to x)}$ for each $v\in\mathcal{V}(n)$. The latter operad is known to be Koszul \cite[Lemma~55]{MV}, so the homology of the cobar complex of its Koszul dual is concentrated in degree zero. Thus further differentials of our spectral sequence vanish, and the  homology of $\Omega(\mathcal{V}_{\bullet\to\bullet,\infty})$ is concentrated in degree zero, where it is, by direct inspection, equal to $\mathcal{P}_{\bullet\to\bullet}$.
\end{proof}

\subsection{Homotopy transfer theorem for homotopy cooperads}

One of the key features of homotopy structures is that they can be transferred along homotopy retracts. The following result generalising (and dualising) both the homotopy transfer formulae for $A_\infty$-coalgebras \cite{KS,LV} and the homotopy transfer formulae for (pr)operads \cite{Gra} is proved in~\cite{DV}. The signs in the formulae are Koszul signs coming from various (de)suspensions, and writing them by a closed formula is not in any way useful; see \cite{Gra} for some further explanations of the origin of signs. 

\begin{proposition}[\cite{DV}]\label{thm:HTTCoop}
Let $(\mathcal{C}, \lbrace \Delta_t \rbrace)$ be a homotopy cooperad. Let
$(\mathcal{H}, d_\mathcal{H})$ be a dg $\mathbb{S}$-module, which is a homotopy retract of the dg $\mathbb{S}$-module $(\mathcal{C}, d_\mathcal{C})$:
\begin{eqnarray*}
&\xymatrix{     *{ \quad \ \  \quad (\mathcal{C}, d_\mathcal{C})\ } \ar@(dl,ul)[]^{H}\ \ar@<0.5ex>[r]^-{p} & *{\
(\mathcal{H}, d_\mathcal{H}) \ .\quad \ \  \ \quad }  \ar@<0.5ex>[l]^-{i}} &
\end{eqnarray*}
Consider the formulae
\begin{equation}\label{HTTforCoops}
\widetilde{\Delta}_t:=\sum \pm\,   t(p)\circ \big(
(\Delta_{t_{k+1}} H)     \circ_{j_k}   (\cdots   (\Delta_{t_3} H)     \circ_{j_2}  ((\Delta_{t_2} H)  \circ_{j_1} \Delta_{t_1} )   )
     \big)\circ i \ ,
\end{equation}
where $t$ is a tree with at least two vertices, and the sum is over all possible ways of writting it by successive expansions of trees with at least two vertices,
 \[
t=(((t_1\circ_{j_1} t_2)\circ_{j_2} t_3) \cdots )\circ_{j_k} t_{k+1}    \ ,
 \]
so one begins with the tree $t_{1}$, expands its vertex $j_1$ by replacing it with the tree $t_2$, then expands the vertex $j_2$ of the result by replacing it with the tree $t_3$ etc.  (The notation $(\Delta_{t'} H)  \circ_j \Delta_{t}$ means that we apply $\Delta_{t'} H$ at the $j^\textrm{th}$ vertex of the $t$-shaped elements of the free operad arising upon the application of $\Delta_{t}$). These formulae create the necessary ``correction terms'' one has to add to the transferred decomposition maps $t(p)\circ \Delta_t\circ i$ in order to define a homotopy cooperad structure on the dg $\mathbb{S}$-module $(\mathcal{H}, d_\mathcal{H})$. 
\end{proposition}

\subsection{Maurer--Cartan description of homotopy algebras and morphisms}\label{ho-coop}

Here we discuss, following \cite{MV,vdL}, a description of homotopy $\mathcal{P}$-algebras and homotopy morphisms of those algebras in terms of solutions to the Maurer--Cartan equation in a certain $L_\infty$-algebra. Unlike the case of differential graded Lie algebras, the defining equation of Maurer--Cartan elements in $L_\infty$-algebras only makes sense under some extra conditions; we recall one of the possible choices, following \cite{Ber11}.

\begin{definition}
A $L_\infty$-algebra $L$ is said to be \emph{complete} if it is equipped with a decreasing filtration 
 \[
L=F^1L\supseteq F^2L\supseteq \ldots\supseteq F^nL\supseteq\ldots
 \]
such that 
\begin{itemize}
 \item for each $k$ and $r$,  we have \[\ell_k(F^rL,L,\ldots,L)\subseteq F^{r+1}L\]
 \item for each $r$, there exists some $N$ such that for all $k>N$ we have \[\ell_k(L,L,\ldots,L)\subset F^rL\]
 \item $L$ is complete with respect to this filtration, that is the canonical map \[L\to \varprojlim L/F^rL\] is an isomorphism.
\end{itemize}
A complete $L_\infty$-algebra $L$ is said to be profinite if each quotient $L/F^rL$ is finite dimensional.
\end{definition}

A class of complete $L_\infty$-algebras that we shall primarily need for our purposes is given by \emph{convolution $L_\infty$-algebras}. Suppose that $\mathcal{C}$ is a homotopy cooperad with the total decomposition map
  \[
\Delta_{\mathcal{C}}\colon s^{-1}\mathcal{C}\to\mathcal{F}(s^{-1}\mathcal{C})^{(\ge 2)},
  \] and $\mathcal{P}$ is a dg operad with the induced composition map
  \[
\widetilde{\mu}_{\mathcal{P}}\colon\mathcal{F}(\mathcal{P})^{(\ge 2)}\to \mathcal{P}.
 \]
(Note that we use the (de)suspended definition in one case, and the usual definition in the other one; this corresponds to the almost-self-duality of the coloured operad encoding non-coloured operads). In this case the collection $\Hom_{\k}(\mathcal{C},\mathcal{P})$ is a homotopy operad, the \emph{convolution homotopy operad of $\mathcal{C}$ and $\mathcal{P}$}, and hence the product of components of this collection is an $L_\infty$-algebra~\cite{vdL}. The structure maps $\ell_n$ of that $L_\infty$-algebra are, for $n>1$,
\begin{equation}\label{l-infty}
\ell_n(\phi_1,\ldots,\phi_n)=\sum_{\sigma\in\mathbb{S}_n}(-1)^{\sgn(\sigma,\phi_1,\ldots,\phi_n)}\widetilde{\mu}_{\mathcal{P}}\circ(\phi_{\sigma(1)}\otimes\ldots\otimes \phi_{\sigma(n)})\circ(s^{\otimes n})\circ\Delta_n s^{-1},
\end{equation}
where $\Delta_n$ is the component of ${\Delta}_{\mathcal{C}}$ which maps $\mathcal{C}$ to $\mathcal{F}(\overline{\mathcal{C}})^{(n)}$, that is the sum of all cooperations $\Delta_t$ over trees $t$ with $n$ internal vertices, see~\cite{MV,vdL}. The map $\ell_1$ is the usual differential of the space of maps between two chain complexes:
 \[
\ell_1(\phi)=D(\phi)=d_{\mathcal{P}}\circ \phi-(-1)^{|\phi|}\phi\circ d_{\mathcal{C}}.
 \]
The product of the spaces of $\mathbb{S}_n$-equivariant maps 
 \[
  \prod_{n\ge 1} \Hom_{\mathbb{S}_n}(\mathcal{C}(n),\mathcal{P}(n))
 \]
can be shown to be an $L_\infty$-subalgebra of this algebra, which we shall be referring to as \emph{convolution $L_\infty$-algebra} of $\mathcal{C}$ and $\mathcal{P}$. 

All $L_\infty$-algebras we consider in this paper will arise as convolution algebras. To ensure their completeness, we shall be using the following result (which, in all cases we deal with, will be manifestly applicable). 

\begin{proposition}
If the cobar complex of $\mathcal{C}$ is a minimal operad with finite-dimensional components $\mathcal{C}^{(k)}$ of the decomposition  $\mathcal{C}=\bigoplus_{k\ge 1}\mathcal{C}^{(k)}$ implementing the Sullivan triangulation condition, then $L$, the convolution $L_\infty$-algebra of $\mathcal{C}$ and $\mathcal{P}$ is a complete $L_\infty$-algebra with respect to the filtration whose $p^\text{th}$ term $F^pL$ is given by 
 \[
\prod_{n\ge1, k\ge  p} \Hom_{\mathbb{S}_n}(\mathcal{C}^{(k)}(n),\mathcal{P}(n)), 
 \]
that is the maps that vanish on $\bigoplus_{k<p}\mathcal{C}^{(k)}$. 
\end{proposition}

\begin{proof}
The first condition of completeness follows directly from the Sullivan triangulation condition: the operadic decomposition of an element from $\mathcal{C}^{(k)}$ with $k\le r$ does not contain elements from $\mathcal{C}^{(r)}$ and higher, so $\ell_k(F^rL,L,\ldots,L)\subseteq F^{r+1}L$. The second condition essentially expresses the fact that for every $c\in\mathcal{C}$ the number of trees $t$ for which $\Delta_t(c)\ne 0$ is finite. The third condition is obvious.  \qed
\end{proof}

\begin{definition}
Let $L$ be a complete $L_\infty$-algebra with the structure maps $\ell_k$, $k\ge 1$. An element $\alpha\in L_{-1}$ is said to be a Maurer--Cartan element (notation: $\alpha\in\MC(\frg)$) if
\[
\sum_{k\ge 1}\frac{1}{k!}\ell_k(\alpha,\alpha,\ldots,\alpha)=0.
\]
\end{definition}

Note that the Maurer--Cartan equation in a complete $L_\infty$-algebra makes sense, since the infinite series converges with respect to topology defined by the filtration. In all formulae throughout this paper, we only use infinite series in $L_\infty$-algebras that are complete, and no convergence issues ever arise.

Let $L$ be a complete $L_\infty$-algebra with the structure maps $\ell_k$, $k\ge 1$, and let $\alpha$ be a Maurer--Cartan element of that algebra. One can consider the following new operations on~$L$:    
 \[
\ell^\alpha_n(x_1,\ldots,x_n):=\sum_{p\ge 0} \frac{1}{p!} \ell_{n+p}(\underbrace{\alpha, \ldots, \alpha}_p, x_1,\ldots, x_n).   
 \] 
It is known~\cite[Prop.~4.4]{Get} that the underlying vector space of $L$ equipped with the structure maps $\ell^\alpha_k$, $k\ge 1$, is again a complete $L_\infty$-algebra, denoted by $L^\alpha$. The $L_\infty$-structure of that algebra is sometimes called the \emph{$L_\infty$-structure twisted by $\alpha$}. 

\smallskip

Let us consider the $\{x,y\}$-coloured homotopy cooperad $\mathcal{V}_{\bullet\to\bullet, \infty}$ from Section~\ref{infinitymor}, and $\{x,y\}$-coloured operad 
$\End_{X_x\oplus Y_y}$. The general construction of Section~\ref{ho-coop} produces an $L_\infty$-algebra structure on the space of $\mathbb{S}$-module morphisms
 \[
\mathcal{L}_{X,Y}:=\Hom_{\mathbb{S}}(\mathcal{V}_{\bullet\to\bullet, \infty},\End_{X_x\oplus Y_y})
 \]
between them. This space of morphisms can be naturally identified with the space
\[
(h_x,h_y,h_{xy})\in\Hom_\k(\overline{\mathcal{P}^{\ac}}(X),X)\oplus\Hom_\k(\overline{\mathcal{P}^{\ac}}(Y),Y)\oplus\Hom_\k(s\mathcal{P}^{\ac}(X),Y),
\]
and in what follows we shall view this latter space as the underlying space of the $L_\infty$-algebra $\mathcal{L}_{X,Y}$. 

The following is proved in \cite{MV} for properads, and is essentially present in \cite{KS} in the case of operads.

\begin{proposition}
A triple of elements $(h_x,h_y,h_{xy})$ of the vector space
 \[
\Hom_\k(\overline{\mathcal{P}^{\ac}}(X),X)\oplus\Hom_\k(\overline{\mathcal{P}^{\ac}}(Y),Y)\oplus\Hom_\k(s\mathcal{P}^{\ac}(X),Y)
 \]
is a solution to the Maurer--Cartan equation of the $L_\infty$-algebra $\mathcal{L}_{X,Y}$ if and only if $h_x$ is a structure of a homotopy $\mathcal{P}$-algebra on $X$, $h_y$ is a structure of a homotopy $\mathcal{P}$-algebra on $Y$, and $h_{xy}$ is a homotopy morphism between these algebras.
\end{proposition}

Moreover, for two given homotopy $\mathcal{P}$-algebra structures on $X$ and $Y$, it is possible to describe homotopy morphisms between the corresponding algebras in the same way. Suppose that $X$ and $Y$ are two homotopy $\mathcal{P}$-algebras, so that the algebra structures are enconded by the elements $h_x\in\Hom_\k(\overline{\mathcal{P}^{\ac}}(X),X)$ and $h_y\in\Hom_\k(\overline{\mathcal{P}^{\ac}}(Y),Y)$ respectively. Since the zero map is manifestly a homotopy morphism, the triple $\alpha=(h_x,h_y,0)$ is a Maurer--Cartan element of the $L_\infty$-algebra
$\mathcal{L}_{X,Y}$.

\begin{proposition}\label{def-of-L}
In the twisted $L_\infty$-algebra $\mathcal{L}_{X,Y}^\alpha$, the subspace
 \[
\mathcal{L}(X,Y):= \Hom_\k(s\mathcal{P}^{\ac}(X),Y)
 \]
is an $L_\infty$-subalgebra. Solutions to the Maurer--Cartan equation in that subalgebra are in one-to-one correspondence with homotopy morphisms between $X$ and $Y$.
\end{proposition}

\begin{proof}
First, using the homotopy cooperad structure on $\mathcal{V}_{\bullet\to\bullet, \infty}$, one can see by direct inspection that if we put $\beta=(0,0,h_{xy})\in \mathcal{L}_{X,Y}$, then the element $\alpha+\beta$ is a Maurer--Cartan element of $\mathcal{L}_{X,Y}$ if and only if $\beta$ is a Maurer--Cartan element of $\mathcal{L}_{X,Y}^\alpha$. Therefore, if we check that $\mathcal{L}(X,Y)$ is an $L_\infty$-subalgebra of $\mathcal{L}_{X,Y}^\alpha$, the statement follows. In fact, it is possible to show that $\mathcal{L}(X,Y)$ is an ideal of $\mathcal{L}_{X,Y}$, that is $\ell_k(x_1,\ldots,x_k)$ is in $\mathcal{L}(X,Y)$ whenever at least one of the arguments is. Indeed, the decomposition maps of $\overline{\mathcal{P}^{\ac}}_{(x\to x)}$ and $\overline{\mathcal{P}^{\ac}}_{(y\to y)}$ do not produce elements from $s\mathcal{P}^{\ac}_{(x\to y)}$, therefore the first two components of the element $\ell_k(x_1,\ldots,x_k)$ of
 \[
\Hom_\k(\overline{\mathcal{P}^{\ac}}(X),X)\oplus\Hom_\k(\overline{\mathcal{P}^{\ac}}(Y),Y)\oplus\Hom_\k(s\mathcal{P}^{\ac}(X),Y)
 \]
vanish whenever at least one of the $x_i$ is in $\mathcal{L}(X,Y)$. This implies that $\mathcal{L}(X,Y)$ is a subalgebra of the twisted algebra, since the twisted operations are made up of the original ones, and in each term at least one of the arguments belongs to the subspace $\mathcal{L}(X,Y)$.\qed
\end{proof}

It is easy to use our formulae to obtain explicit formulae for the structure maps of the $L_\infty$-algebra $\mathcal{L}(X,Y)$. Its differential is given by the formula
\begin{equation}\label{Lxy-diff}
\ell_1(\phi)(sx)=(d^{(1)}_Y\circ\phi)(sx)+(-1)^{|\phi|} (\phi\circ sD_X)(x) ,
\end{equation}
and for $k>1$, the structure maps $\ell_k$ are given by
\begin{equation}\label{Lxy-brackets}
\ell_k(\phi_1,\ldots,\phi_k)(sx)=\sum_{\sigma\in\mathbb{S}_k} (d^{(k)}_Y\circ(\id\otimes\phi_{\sigma(1)}\otimes\cdots\otimes\phi_{\sigma(k)})\circ (1\otimes s^{\otimes k})\circ \Delta_X^{k-1})(x),
\end{equation}
where
 \[
\Delta_X^{k-1}\colon \mathcal{P}^{\ac}(X)\to \mathcal{P}^{\ac}(k)\otimes_{\mathbb{S}_k}\mathcal{P}^{\ac}(X)^{\otimes k}
 \]
is the $k^\text{th}$ cooperation in the cofree $\mathcal{P}^\ac$-coalgebra $\mathcal{P}^{\ac}(X)$.

\section{Overview of existing notions of homotopies}\label{homotopies}

\subsection{Concordance}

The definition in this section originates from a classical geometric picture: if \[f\colon X\times I\to Y\] is a homotopy connecting the two given manifold maps $p(\cdot)=f(\cdot,0)$ and $q(\cdot)=f(\cdot,1)$ between smooth manifolds $X$ and $Y$, then $p$ and $q$ induce the same map on the cohomology. This is proved by constructing a chain homotopy between $p$ and $q$. Let us briefly recall the way it is done. The map $f$ induces a morphism of de Rham complexes \begin{equation}f^*\colon\Omega^\bullet(Y)\to\Omega^\bullet(X)\otimes\Omega^\bullet(I)\label{tensforms}\end{equation} (if we can work with algebraic differential forms, so that $\Omega^\bullet(X\times I)\simeq\Omega^\bullet(X)\otimes\Omega^\bullet(I)$), and is determined by two maps $f_0,f_1\colon \Omega^\bullet(Y)\to \Omega^\bullet(X)\otimes\Omega^0(I)$ with 
 \[
f^*(c)=f_0(c)+f_1(c)dt  
 \]
for each $c\in\Omega^\bullet(Y)$. Writing down the condition for $f^*$ to be a map of chain completes, we observe that
 \[
(-1)^{|c|}\dot{f_0}(c)\,dt=-d_X(f_1(c))\,dt+ f_1(d_Y(c))\,dt,
 \]
and integrating this equation over $I$ gives
 \[
q^*-p^*=d_X h+h d_Y,
 \]
where $h(c)=(-1)^{|c|-1}\int_I f_1(c)\,dt$. 

It is very natural to apply a similar approach to homotopy algebras. Note that tensoring with $\Com$ does not change the operad, so, for example, if a chain complex $V$  has a structure of a homotopy $\mathcal{P}$-algebra, the tensor product $V\otimes \Omega^\bullet(I)$ is a homotopy $\mathcal{P}$-algebra as well. For each structure map $\lambda$, we have  
 \[
\lambda(v_1\otimes\omega_1,\ldots,v_n\otimes\omega_n)=\pm\lambda(v_1,\ldots,v_n)\otimes (\omega_1\wedge\cdots\wedge\omega_n) ,
 \] 
with the sign determined by the Koszul sign rule. In what follows, we develop this idea, denoting $\Omega^\bullet(I)$ by $\Omega$ for brevity.

\begin{definition}\label{Concordance}
Two homotopy morphisms $p,q$ between two homotopy $\mathcal{P}$-algebras $X,Y$ are said to be \emph{concordant} if there exists a morphism $\phi$ of dg $\mathcal{P}^\ac$-coalgebras
 \[
\phi\colon \mathsf{B}(X)\to \mathsf{B}(Y\otimes\Omega)
 \]
for which $p(v)=\phi(v)|_{t=0}$ and $q(v)=\phi(v)|_{t=1}$ whenever $v\in\mathsf{B}(X)$. 
\end{definition}

\begin{remark}
One of the first fundamental results of formal deformation theory states that for an $L_\infty$-algebra $L$ and Artinian local algebra~$A$, there exists a bijection between the set of Maurer--Cartan elements of the $L_\infty$-algebra $L\otimes A$ and the set of all dg coalgebra morphisms from $\mathfrak{m}_A^*$ to the bar complex $\mathsf{B}(L)$ (see, e.g., Drinfeld's letter to Schechtman on deformation theory~\cite{Dr88}). In a sense, the notion of concordance may be thought as an attempt to use this definition with $A$ being the dg algebra $\Omega$, which however is not Artinian so various precautions and reformulations are required. 
\end{remark}

\begin{remark}
In the case $\mathcal{P}=\Lie$, this definition of concordance is closely related to that from~\cite{SSS}. The main difference is that there cobar complexes are used, and hence one needs to dualise algebras in question. In the case of infinite dimensional vector spaces, this would create various technical problems, and hence we chose to alter the definition. In our case, such a map is determined by its corestriction $\mathsf{B}(X)\to Y\otimes\Omega$ which has to satisfy a certain equation (compatibility with differentials), while in~\cite{SSS}, concordance is defined via a map of cobar complexes $\Omega(Y^*)\to \Omega(X^*)\otimes\Omega$ (reminiscent of the ``geometric'' map above) which is determined by its restriction $Y^*\to \Omega(X^*)\otimes\Omega$ subject to compatibility with differentials. One easily checks that if the algebras $X$ and $Y$ are finite-dimensional, then in both definitions the data involved and the conditions on that data are exactly the same. 
\end{remark}

\subsection{Homotopy of Maurer--Cartan elements of $L_\infty$-algebras}

In this section, we outline the notions of homotopy between Maurer--Cartan elements of homotopy Lie algebras. 

\smallskip

\noindent
\textbf{Warning. } We would like remind the reader that the letter $L$ always denotes a complete $L_\infty$-algebra. For such an algebra, we shall use, on several occasions, notation like $L[t]=L\otimes\k[t]$, or more generally $L\otimes A$, where $A$ is some finitely generated differential graded algebra. In such cases, we shall implicitly mean that instead of those spaces we shall work with their completions with respect to the filtration derived from the filtration on $L$ for which $L$ is complete. (If $L$ is nilpotent as in \cite{Get}, then no such completion would of course be needed, but for $L$ complete  it is necessary). 

\smallskip

One available approach to equivalence of Maurer--Cartan elements is inspired by rational homotopy theory. Namely, if one considers a simplicial differential graded commutative associative algebra $\Omega_\bullet$ whose $n$-simplices are differential forms on the $n$-simplex $\Delta^n$, then one can prove, under appropriate finiteness assumptions, that for a differential graded Lie algebra $L$, the set of homomorphisms of differential graded algebras from the cohomological Chevalley--Eilenberg complex $C^*(L)$ to $\Omega_n$ is naturally identified with Maurer--Cartan elements of $L\otimes\Omega_n$. This suggests to introduce a simplicial set $\MC_\bullet(L)$ by the formula
  $$
 \MC_\bullet(L)=\MC(L\otimes\Omega_\bullet),
  $$
and that set is in some sense is the main protagonist of rational homotopy theory, connecting homotopy theory of nilpotent differential graded Lie algebras and that of nilpotent rational topological spaces. 

In \cite{Get}, Getzler proposed to study a simplicial set $\gamma_\bullet(L)$ which is smaller than $\MC_\bullet(L)$ but carries the same homotopy information. The main ingredient in his construction is the Dupont\rq{}s \cite{Dup76} chain homotopy $s_\bullet\colon\Omega_\bullet^\bullet\to\Omega_\bullet^{\bullet-1}$; by definition,
  $$
 \gamma_\bullet(L):=\{\alpha\in\MC_\bullet(L)\colon s_\bullet(\alpha)=0\}.
  $$

\paragraph{Quillen homotopy.}

The notion of Quillen homotopy equivalence of Maurer--Cartan elements also uses the de Rham algebra $\Omega=\Omega_1=\Omega^\bullet(I)$ and its two evaluation morphisms $\phi_s\colon(\Omega,d)\to(\k,0)$, $s\in\{0,1\}$, given by $\phi_s(t)=s$, where $t$ is, as above, the coordinate in $I$. 
The motivation for this definition is geometric: if $L$ is a model of a pointed space $Y$ in the sense of rational homotopy theory, then, as pointed out in~\cite{BM2}, $L\otimes\Omega$ is a model of the evaluation fibration $ev\colon\mathop{\mathrm{map}}^*(I,Y)\to Y$, $ev(\gamma)=\gamma(1)$.

\begin{definition}
Two elements $\alpha_0,\alpha_1\in\MC(L)$ are said to be \emph{Quillen homotopic} if there exists $\beta\in\MC(L\otimes\Omega)$ for which $\phi_0(\beta)=\alpha_0$, $\phi_1(\beta)=\alpha_1$.
\end{definition}

\paragraph{Gauge homotopy.}

The set $\MC(L)$ under appropriate finiteness assumptions acquires a structure of a scheme, see \cite{Sch98}. It is well understood that the right notion of ``gauge symmetries'' of $\MC(L)$, for $L$ being a dg Lie algebra, is given by the group associated to the Lie algebra~$L_0$, see \cite{GM} for details. So it is natural to look for a similar concept in the general case of $L_{\infty}$-algebras. The corresponding theory was systematically developed by Getzler~\cite{Get}. Application of these $L_\infty$-gauge symmetries to studying homotopies between morphisms of $L_{\infty}$-algebras goes back to~\cite{Dol07}.

The following statement is contained in~\cite{Get}; however, there it is a consequence of much more general results, so for the convenience of the reader we present a more hands-on proof.

\begin{proposition}
Let $L$ be an $L_\infty$-algebra, and $x\in L_0$. The vector field $V_x$ on $L_{-1}$ defined by
 \[
V_x(\alpha)=-\ell^\alpha_1(x)
 \]
is a tangent vector field of the set of Maurer--Cartan elements of $L$.
\end{proposition}

\begin{proof} Note that the tangent vectors $\beta\in L_{-1}$ to $\MC(L)$ at a point $\alpha$ are characterized by
 \[
\sum_{p\ge0}\frac{1}{p!}\ell_{p+1}(\underbrace{\alpha,\ldots,\alpha,}_{p \text{ times}}\beta)=0,
 \]
that is
\[
\ell^\alpha_1(\beta)=0.
\]
The value of $V_x$ at $\alpha$ satisfies this condition since
\[
\ell^\alpha_1(V_x(\alpha))=\ell^\alpha_1(-\ell^\alpha_1(x))=-(\ell^\alpha_1)^2(x)=0,
\]
which completes the proof.\qed
\end{proof}

Let $\alpha$ be a Maurer--Cartan element of $L$, and consider the the integral curve $\alpha(t)$ of $V_x$ starting at $\alpha$, that is the solution of the differential equation
\[
\frac{d\alpha}{dt}+\ell^\alpha_1(x)=0
\]
satisfying the initial condition $\alpha(0)=\alpha$. (This solution is an element of $L[t]$ (completed as usual), an explicit formula for it is given in \cite[Prop.~5.7]{Get}). The previous result implies the following statement. 

\begin{corollary}
We have $\alpha(t)\in\MC(L[t])$. Also, for each $t$ the element $\alpha(t)$ is an element of $\MC(L)$.
\end{corollary}

This suggests a meaningful definition of gauge homotopy.

\begin{definition}
Two elements $\alpha_0,\alpha_1\in\MC(L)$ are said to be \emph{gauge homotopic} if for some $x\in L_0$ there exists an integral curve $\alpha(t)$ of $V_x$ with $\alpha(0)=\alpha_0$ and $\alpha(1)=\alpha_1$.
\end{definition}

In Section~\ref{Quillen-and-gauge} below, we shall explain why two Maurer--Cartan elements are gauge homotopic if and only if they are Quillen homotopic. That was proved in \cite{Man1} for dg Lie algebras. In~\cite{Dol07}, this statement is needed in the full generality for $L_\infty$-algebras; however, the proof given there formally proves a somewhat weaker statement, so we fill that gap here rather than merely referring the reader to \cite{Dol07}. 

\paragraph{Cylinder homotopy.}

The main motivation for the definition of this section is as follows. Consider the quasi-free dg Lie algebra $\mathfrak{l}$ with one generator $x$ of degree $-1$ and the differential $d$ given by $dx=-\frac{1}{2}[x,x]$.
Note that for a dg Lie algebra $L$ the set of Maurer--Cartan elements can be identified with the set of dg Lie algebra morphisms from $\mathfrak{l}$ to $L$. Thus, if in the homotopy category of dg Lie algebras we can come up with a cylinder object for $\mathfrak{l}$, the homotopy relation for Maurer--Cartan elements can be defined using that cylinder. It turns out that a right cylinder is given by the \emph{Lawrence--Sullivan construction}.

The Lawrence--Sullivan Lie algebra $\mathfrak{L}_{LS}$ is a (pronilpotent completion of a) certain quasi-free Lie algebra, that is, a free graded Lie algebra with a differential $d$ of degree $-1$ satisfying $d^2=0$ and the Leibniz rule. It is freely generated by the elements $a,b,z$, where $|a|=|b|=-1$, $|z|=0$, and
\begin{gather*}
da+\frac12[a,a]=db+\frac12[b,b]=0,\\
dz=[z,b]+\sum_{k\ge 0} \frac{B_k}{k!}\ad_z^k(b-a)=\ad_z(b)+\frac{\ad_z}{\exp(\ad_z)-1}(b-a),
\end{gather*} where the $B_k$ are the Bernoulli numbers.
It is indeed shown in \cite{BM2} that this algebra gives the right cylinder object for $\mathfrak{l}$ in the homotopy category of dg Lie algebras, hence the following definition.

\begin{definition}
Two elements $\alpha_0,\alpha_1\in\MC(L)$ are said to be \emph{cylinder homotopic} if there exists an $L_\infty$-morphism from $\mathfrak{L}_{LS}$ to $L$ which takes $a$ to~$\alpha_0$ and $b$ to~$\alpha_1$.
\end{definition}

It turns out that the arising notion of homotopy for Maurer--Cartan elements is equivalent to the other ones available.

\begin{proposition}[{\cite[Prop.~4.5]{BM1}}]
Two Maurer--Cartan elements of an $L_\infty$-algebra are cylinder homotopic if and only if they are Quillen homotopic.
\end{proposition}

In what follows, we shall use as a toy example the homotopy coassociative algebra $A_{LS}$ that corresponds to the differential on the universal enveloping algebra of~$\mathfrak{L}_{LS}$. This $A_\infty$ coalgebra is defined on the linear span of the elements $u=sa$, $v=sb$, $w=sz$, where $|u|=|v|=0$, $|w|=1$, and is explicitly given by
\begin{gather*}
\delta_1(w)=u-v, \quad \delta_1(u)=\delta_1(v)=0,\\
\delta_2(w)=-\frac12w\otimes(u+v)-\frac12(u+v)\otimes w,\quad \delta_2(u)=-u\otimes u, \quad \delta_2(v)=-v\otimes v,\\
\delta_k(w)=-\sum_{p+q=k-1}\frac{b_{k-1}}{p!q!}w^{\otimes p}\otimes(u-v)\otimes w^{\otimes q},\quad \delta_k(u)=\delta_k(v)=0, k\ge 3.
\end{gather*}

\subsection{Operadic homotopy}

Let us recall the operadic approach to homotopies between homotopy morphisms \cite{Doubek,Mar2}. Recall that homotopy $\mathcal{P}$-algebras are algebras over operad $\mathcal{P}_\infty=\Omega(\mathcal{P}^\ac)$, the minimal model of the operad $\mathcal{P}$, and homotopy morphisms between homotopy algebras are algebras over the minimal model $\mathcal{P}_{\bullet\to\bullet,\infty}$ of the coloured operad  $\mathcal{P}_{\bullet\to\bullet}$ encoding morphisms of $\mathcal{P}$-algebras. One hopes to include these results in a hierarchy of results that would incorporate higher homotopies as well, but the situation is somewhat subtle. 

In \cite{Mar2}, an operadic approach to homotopies between morphisms is outlined. Let us state a version of that approach which is inspired by \cite[Th.~18]{Mar2}. Our formulae are different in two ways. First, we restrict ourselves to the case of a Koszul operad $\mathcal{P}$, and as a consequence are able to make some formulae more precise. Second, we work with homotopy cooperads as opposed to quasi-free operads, therefore some (de)suspensions make signs in our formulae differ from those of \cite{Mar2}, and the differential is separated from the decomposition maps. 

\begin{definition}
We say that the operad $\mathcal{P}$ satisfies the \emph{homotopy hypothesis} if  there exists a quasi-free resolution of the operad $\mathcal{P}_{\bullet\to\bullet}$ of the form $\Omega(\mathcal{V}_{\bullet\rightrightarrows\bullet,\infty})$, where the homotopy cooperad  $\mathcal{V}_{\bullet\rightrightarrows\bullet,\infty}$ has the underlying chain complex 
\[
\mathcal{V}_{\bullet\rightrightarrows\bullet,\infty}:=\overline{\mathcal{P}^{\ac}}_{(x\to x)}\oplus \overline{\mathcal{P}^{\ac}}_{(y\to y)}\oplus s\mathcal{P}^{\ac}_{(x\to y)}\otimes C_\bullet([0,1]) ,
\] 
and the decomposition maps 
\begin{itemize}
	\item[-] induced by infinitesimal decomposition maps of the cooperad $\overline{\mathcal{P}^{\ac}}$ on both $\overline{\mathcal{P}^{\ac}}_{(x\to x)}$ and $\overline{\mathcal{P}^{\ac}}_{(y\to y)}$
	\item[-] induced by the homotopy cooperad structure of the $\mathcal{V}_{\bullet\to\bullet,\infty}$ on both $\overline{\mathcal{P}^{\ac}}_{(x\to y)}\otimes\mathbf{0}$ and $\overline{\mathcal{P}^{\ac}}_{(x\to y)}\otimes\mathbf{1}$
	\item[-] on elements $sM_{(x\to y)}\otimes\mathbf{01}\in s\mathcal{P}^{\ac}_{(x\to y)}\otimes\mathbf{01}$ corresponding to an element $M\in\mathcal{P}^\ac(n)$ have the shape
	\begin{equation}
	\Delta(sM_{(x\to y)}\otimes\mathbf{01})=\mathbf{01}\circ(M_{(x\to x)})-(M_{(y\to y)})\circ [[\mathbf{01}]]_n+\cdots, \label{HomotopyLT}
	\end{equation}
	where
	\[
	[[\mathbf{01}]]_n=\Sym\left(\mathbf{01}\otimes\mathbf{0}^{\otimes(n-1)}+\mathbf{1}\otimes\mathbf{01}\otimes \mathbf{0}^{\otimes(n-2)}+\cdots+\mathbf{1}^{\otimes(n-1)}\otimes\mathbf{01}\right) ,
	\]
	and the ``non-leading terms'' (denoted by `$\cdots$' in the formula above) belong to the ideal generated by the operations $s\overline{\mathcal{P}^{\ac}}_{(x\to y)}\otimes C_\bullet([0,1])$ of arity strictly less than $n$ (that justifies referring to them as non-leading terms).  
\end{itemize} 
\end{definition}	

Assuming that the homotopy hypothesis is satisfied, it is natural to define homotopy between two homotopy morphisms of two homotopy $\mathcal{P}$-algebras as a $\Omega(\mathcal{V}_{\bullet\rightrightarrows\bullet,\infty})$-algebra whose structure maps from $\overline{\mathcal{P}^{\ac}}_{(x\to x)}$ and $\overline{\mathcal{P}^{\ac}}_{(y\to y)}$ define the given homotopy $\mathcal{P}$-algebra structures, structure maps from $\overline{\mathcal{P}^{\ac}}_{(x\to y)}\otimes\mathbf{0}$ and $\overline{\mathcal{P}^{\ac}}_{(x\to y)}\otimes\mathbf{1}$ define the given morphisms. Essentially, the structure map $\overline{\mathcal{P}^{\ac}}_{(x\to y)}\otimes\mathbf{01}$ provides a homotopy between the morphisms. 

\smallskip

The claim of \cite[Th.~36]{Mar2} essentially implies that the homotopy hypothesis holds.  In \cite{Doubek}, it is noted that the proof of the above result given in \cite{Mar2} is incomplete, and a proof of a much more general statement under, however, somewhat more restrictive assumptions on the operad $\mathcal{P}$ is given.  In Section~\ref{concordance-and-Markl}, we shall explain how to view this result from a different angle and prove the homotopy hypothesis for any Koszul operad.

\begin{remark}
It is worth noticing that even if the homotopy hypothesis holds,  such a model of $\mathcal{P}_{\bullet\to\bullet}$ does not have to be unique (it is not minimal by the construction, and hence there is freedom in how to reconstruct the 
non-leading terms). In one example, the nonsymmetric operad of associative algebras, explicit formulae for images of the generators under the differential were computed in \cite{Mar2}, and it was observed that this recovers the definition of a natural transformation between two $A_\infty$-functors based on the notion of derivation homotopy, as in~\cite{Fuk,Gran99,L-H,Lyu}.	
\end{remark}	

\begin{remark} 
The formula for $[[h]]$ makes one think of derivation homotopies as well, but this intuition is only correct under very restrictive assumptions, e.g., for the case $\mathcal{P}=\Lie$ the derivation homotopy formulae only work for Abelian $L_\infty$ algebras, see \cite{SSS}. Nonetheless the derivation homotopy formulae do always work for nonsymmetric operads. The reason for that is that one can use the \v{C}ech cochain complex $C_\bullet([0,1])^\vee$ in the place of the de Rham complex $\Omega$, since  tensoring with an associative algebra does not change the type of algebras for algebras over nonsymmetric operads. The formulae
 \[
e_{\mathbf{0}}e_{\mathbf{01}}=e_{\mathbf{01}}=e_{\mathbf{01}}e_{\mathbf{1}}  
 \]
for computing products in the \v{C}ech cochain complex of the interval naturally lead to derivation homotopies.
\end{remark}

\section{Concordance and homotopy of Maurer--Cartan elements}\label{MC-and-Quillen}

\subsection{Concordance and Quillen homotopy}\label{concordance-and-Quillen}

\begin{theorem}\label{Th:Concordance=Quillen}
Two homotopy morphisms of homotopy $\mathcal{P}$-algebras $X$ and $Y$ are concordant if and only if the corresponding Maurer--Cartan elements of $\mathcal{L}(X,Y)$ are Quillen homotopic.
\end{theorem}

\begin{proof}
By definition, $p$ and $q$ are concordant if there exists a morphism of dg $\mathcal{P}^\ac$-coalgebras
 \[
\mathsf{B}(X)\to \mathsf{B}(Y\otimes\Omega)
 \]
that evaluates to $p$ and $q$ at $t=0$ and $t=1$ respectively. The set
 \[
\Hom_{dg \mathcal{P}^\ac\text{-}coalg}(\mathsf{B}(X),\mathsf{B}(Y\otimes\Omega)) 
 \]
is a subset of the space of degree~$0$ linear maps
 \[
\Hom_{\k}(\mathsf{B}(X),Y\otimes\Omega) 
 \]
since a $\mathcal{P}^\ac$-coalgebra morphism from a conilpotent coalgebra to a cofree conilpotent coalgebra is completely determined by its corestrictions on cogenerators, and the compatibility with differentials imposes some conditions that cut out a subset of the space of linear maps.

On the other hand, two Maurer--Cartan elements $\alpha_0,\alpha_1$ of $\MC(\mathcal{L}(X,Y))$ are Quillen homotopic if there exists $\beta\in\MC(\mathcal{L}(X,Y)\otimes\Omega)$ that evaluates to $\alpha_0$ and $\alpha_1$ at $t=0$ and $t=1$ respectively. The set $\MC(\mathcal{L}(X,Y)\otimes\Omega)$ is a subset of the space of degree~$-1$ elements in 
 \[
\mathcal{L}(X,Y)\otimes\Omega\simeq\Hom_\k(s\mathcal{P}^{\ac}(X),Y)\otimes\Omega,
 \]
and the latter can be identified with the space of degree~$0$ elements in
 \[
\Hom_\k(\mathcal{P}^{\ac}(X),Y)\otimes\Omega \simeq \Hom_\k(\mathsf{B}(X),Y)\otimes\Omega\simeq \Hom_\k(\mathsf{B}(X),Y\otimes\Omega).
 \]

We see that the sets we want to identify are embedded into the same vector space. Let us check that the actual equations that define these spaces, that is compatibility with differentials and the Maurer--Cartan equation, actually match. 

Let us take an arbitrary degree~$0$ map $\phi\in\Hom_{\k}(\mathsf{B}(X),Y\otimes\Omega)$. It gives rise to the unique $\mathcal{P}^\ac$-coalgebra morphism $\hat{\phi}$ by the formula
 \[
\hat{\phi}(b)=\sum_{k\ge 1}(\id\otimes\phi^{\otimes k})\circ\Delta_X^{(k-1)}(b).  
 \]
Here $\Delta_X^{(k-1)}$ is the $\mathcal{P}^\ac$-coalgebra decomposition map $\mathsf{B}(X)\to\mathcal{P}^\ac(k)\otimes_{\mathbb{S}_k}(\mathsf{B}(X))^{\otimes k}$; the remaining notation in the following formulas has been already introduced in Section~\ref{sec:operadic-recollection}.
In order for this morphism to be a dg coalgebra morphism, we must have 
\begin{equation}\label{dg-coalgebra-map}
\hat{\phi}\circ D_X=D_{Y\otimes\Omega}\circ \hat{\phi}  .
\end{equation} 
Note that since both $\hat{\phi}\circ D_X$ and $D_{Y\otimes\Omega}\circ \hat{\phi}$ are coderivations of $\mathsf{B}(X)$ with values in $\mathsf{B}(Y\otimes\Omega)$, it is sufficient to check the condition~\eqref{dg-coalgebra-map} on cogenerators, that is only look at its projection on $Y\otimes\Omega$. This way we obtain the condition
\begin{equation}\label{dgc-condition}
(\id\otimes d_{DR})\circ\phi+(d^{(1)}_Y\otimes\id)\circ\phi+\sum_{k\ge 2}(d^{(k)}_{Y\otimes\Omega})\circ (\id\otimes\phi^{\otimes k})\circ\Delta_X^{(k-1)}
=\phi\circ D_X.
\end{equation}

Here $d_{DR}$ is the de Rham differential on $\Omega$. Note that  $d^{(k)}_{Y\otimes\Omega}$ can be explicitly computed as
 \[
d^{(k)}_{Y\otimes\Omega}=\left(d^{(k)}_{Y}\otimes\mu^{(k)}\right)\circ \tau ,   
 \]
where $\tau\colon (Y\otimes\Omega)^{\otimes k}\simeq Y^{\otimes k}\otimes\Omega^{\otimes k}$ is the isomorphism obtained via the structure of symmetric monoidal category of graded vector spaces, and $\mu^{(k)}\colon \Omega^{\otimes k}\to \Omega$ is the $k$-fold product on $\Omega$. 
Let us spell out the Maurer--Cartan condition for a degree $-1$ element 
 \[
\psi\in\Hom_{\k}(s\mathsf{B}(X),Y)\otimes\Omega
 \]
Using explicit formulae \eqref{Lxy-diff}, \eqref{Lxy-brackets} for the $L_\infty$-algebra structure on $\mathcal{L}(X,Y)$ to compute elements in the $L_\infty$-algebra $\mathcal{L}(X,Y)\otimes\Omega$, we rewrite the Maurer--Cartan condition
 \[
\sum_{k\ge 1}\frac{1}{k!}\ell_k(\psi,\ldots,\psi)=0
 \]
evaluated on an element $sb\in s\mathsf{B}(X)$ as follows:
\begin{multline}\label{MC-condition}
(d^{(1)}_Y\otimes\id)\circ(\psi s)(x)-(\psi s)\circ D_X(b)+(1\otimes d_{DR})(\psi s)(x)+\\+\sum_{k\ge2} (d^{(k)}_Y\otimes\mu^{(k)})\circ \tau \circ(\id\otimes(\psi s)^{\otimes k})\circ \Delta_X^{(k-1)}(b)=0.
\end{multline}
It remains to recall that the degree $1$ isomorphism that allows us to identify the graded vector space $\Hom_{\k}(s\mathsf{B}(X),Y)\otimes\Omega$ with $\Hom_{\k}(\mathsf{B}(X),Y\otimes\Omega)$ is essentially given by $\psi\mapsto \phi:=\psi s$, to conclude that the two conditions \eqref{dgc-condition} and \eqref{MC-condition} are the same.\qed
\end{proof}

\subsection{Quillen homotopy and gauge homotopy} \label{Quillen-and-gauge}

The following result is not new; it is nothing but a careful unwrapping of various statements proved in \cite{Get}.
\begin{proposition}
Two elements $\alpha_0,\alpha_1\in\MC(L)$ are Quillen homotopic if and only if they are gauge homotopic.
\end{proposition}

\begin{proof}
Let us begin with an elementary computation. Suppose that $\alpha_0$ and $\alpha_1$ are gauge homotopic. This means that there exists $x\in L_0$ for which an integral curve $\alpha(t)$ of the vector field $V_x$ connects $\alpha_0$ to $\alpha_1$ in the Maurer--Cartan scheme. It satisfies the differential equation $\alpha'(t)+\ell^\alpha_1(x)=0$. 

On the other hand, suppose that $\alpha_0$ and $\alpha_1$ are Quillen homotopic. This means that there exists an element $\beta\in\MC(L\otimes\Omega)$ which evaluates at $\alpha_0$ and $\alpha_1$ at $t=0$ and $t=1$ respectively. Let us write $\beta=\beta_{-1}+\beta_0\,dt$, where $\beta_i\in L[t]$, and the homological degree of $\beta_i$ is $i$. Since $(dt)^2=0$, the Maurer--Cartan equation for $\beta$ becomes
 \[
\beta_{-1}'(t)\,dt +\sum_k\frac1{k!}\ell_k(\underbrace{\beta_{-1},\ldots,\beta_{-1}}_{k \text{ times}})+\sum_k\frac{k}{k!}\ell_k(\underbrace{\beta_{-1},\ldots,\beta_{-1}}_{k \text{ times}},\beta_0)\,dt=0,  
 \]
so $\beta_{-1}\in\MC(L[t])$, and $\beta_{-1}'(t)+\ell_1^{\beta_{-1}}(\beta_0)=0$. 

It follows that for each Quillen homotopy $\beta=\beta_{-1}+\beta_0\,dt$, the pair $(\beta_{-1},\beta_0)$ is almost exactly the datum required for gauge homotopy, with the only difference that $\beta_0$ is an element of $L_0[t]$ rather than $L_0$, as the gauge homotopy would require. (In~\cite{Dol07}, this circumstance is ignored, and a $t$-dependent element~$x$ is obtained, which thus does not literally provide a gauge homotopy). 

Conceptually, the computations we perform merely mean that the Quillen homotopy is precisely the homotopy in the simplicial set  $\MC_\bullet(L)$, and the gauge homotopy is precisely the homotopy in the simplicial set $\gamma_\bullet(L)$, since the Dupont\rq{}s chain homotopy $s_1$ in this case singles out constant $1$-forms. 

Now, recall from~\cite{Get} that the inclusion of simplicial sets $\gamma_\bullet(L)\hookrightarrow\MC_\bullet(L)$ is a homotopy equivalence, in particular $\pi_0(\gamma_\bullet(L))=\pi_0(\MC(L))$. (To be precise, that result applies to nilpotent $L_\infty$-algebras only, and for our purposes an appropriate generalisation of that result for complete $L_\infty$-algebras is required \cite{Ber11}).  The former is precisely given by Maurer--Cartan elements of $L$ modulo the gauge homotopy relation, the latter is given by Maurer--Cartain elements modulo the Quillen homotopy relation. \qed
\end{proof}

\begin{remark}
Our proof, in particular, means that the notions of gauge equivalence and Quillen equivalence coincide in the context of formal deformation theory, where they give rise to equivalence relations on the Maurer--Cartan sets produced by the deformation functor $\MC_L\colon A\mapsto \MC(L\otimes A)$ of Artinian local algebras. For a nice introduction to formal deformation theory, we refer the reader to \cite{Can1,Man1}. 
\end{remark}

\section{Concordance and operadic homotopy} \label{concordance-and-Markl}

\subsection{The concordance operad}\label{concordance-operad}

Let $X$ and $Y$ be two homotopy $\mathcal{P}$-algebras. A $\mathcal{P}^\ac$-coalgebra morphism 
 \[
\phi\colon\mathsf{B}(X)\to \mathsf{B}(Y\otimes\Omega)
 \]
(required to establish the concordance of the morphisms $\phi_{t=0}$ and $\phi_{t=1}$) is completely determined by its corestrictions 
 \[
\mathcal{P}^\ac(X)\to Y\otimes\Omega.
 \]
These corestrictions must in addition determine a dg coalgebra morphism, that is be compatible with the differentials of bar complexes. Let us describe this data operadically. 
Maps 
 \[
\mathcal{P}^\ac(X)\to Y\otimes\Omega
 \]
can be identified with
 \[
\Hom_{\mathbb{S}}(\mathcal{P}^\ac_{(x\to y)}\otimes\Omega^\vee, \End_{X_x\oplus Y_y}),  
 \]
where $\Omega^\vee$ is the \emph{graded} dual coalgebra of $\Omega$, that is a vector space with the basis   
 \[
\alpha_i=(t^i)^\vee, \quad \beta_i=(t^i\,dt)^\vee\quad  (i\ge0),   
 \]
and the coalgebra structure
\begin{gather*}
\delta(\alpha_i)=\sum_{a+b=i}\alpha_a\otimes\alpha_b,\\
\delta(\beta_i)=\sum_{a+b=i}(\beta_a\otimes\alpha_b+\alpha_a\otimes\beta_b).
\end{gather*}
This suggests that the datum of two homotopy $\mathcal{P}$-algebras, and a concordance homotopy between two homotopy morphisms of those algebras may be equivalent to an algebra over a quasi-free coloured operad with generators
 \[
\mathcal{W}=s^{-1}\overline{\mathcal{P}^\ac}_{(x\to x)}\oplus s^{-1}\overline{\mathcal{P}^\ac}_{(y\to y)}\oplus \mathcal{P}^\ac_{(x\to y)}\otimes\Omega^\vee.
 \]
Let us indeed describe such an operad. Equivalently, we shall define a homotopy cooperad on
\[
\mathcal{V}_{\bullet\to\bullet,\Omega}:=s\mathcal{W}=\overline{\mathcal{P}^\ac}_{(x\to x)}\oplus \overline{\mathcal{P}^\ac}_{(y\to y)}\oplus s\mathcal{P}^\ac_{(x\to y)}\otimes\Omega^\vee.
\]

The homotopy cooperad decomposition map on each of the components $\overline{\mathcal{P}^\ac}_{(x\to x)}$ and $\overline{\mathcal{P}^\ac}_{(y\to y)}$ is given by the (honest) cooperad structure on $\mathcal{P}^\ac$, the Koszul dual cooperad of $\mathcal{P}$, each of these components is a sub-cooperad of $s\mathcal{W}$. The homotopy cooperad decomposition of
$s\mathcal{P}^\ac_{(x\to y)}\otimes\Omega^\vee$ does not vanish for trees of two types. The first type is decomposition maps indexed by trees $t$ with two internal vertices; in this case $\Delta_t(s(p\otimes \omega^\vee))$ is obtained from $\Delta_t(p)$ computed in $\mathcal{P}^\ac$ by mapping the root level component isomorphically to $s\mathcal{P}^\ac_{(x\to y)}$ and tensoring the result with $\omega^\vee$, and mapping the other component isomorphically to $\mathcal{P}^\ac_{(x\to x)}$. The second type is decomposition maps indexed by all two-level trees $t$; in this case, assuming that the root of $t$ has $k$ children, $\Delta_t(s(p\otimes \omega^\vee))$ is obtained by applying the full cooperad map 
 \[
\Delta^{(k-1)}\colon\mathcal{P}^\ac\to\mathcal{P}^\ac(k)\otimes_{\mathbb{S}_k}(\mathcal{P}^\ac)^{\otimes k}  
 \]
to $p$, mapping the root level component isomorphically to $\mathcal{P}^\ac_{(y\to y)}$, mapping the other components isomorphically to $s\mathcal{P}^\ac_{(x\to y)}$, and decorating those other components by tensor factors of $\delta^{(k-1)}(\omega^\vee)\in(\Omega^\vee)^{\otimes k}$. These decorations, depending on choices made for writing down representatives of trees, may appear with signs determined from the Koszul sign rule; in addition to that, there is a sign $-1$ appearing globally for all the decompositions of the second type. Finally, the differential of $\mathcal{V}_{\bullet\to\bullet,\Omega}$ is non-zero only on $s\mathcal{P}^\ac_{(x\to y)}\otimes\Omega^\vee$, where it is the dual of the de Rham differential $d_{DR}$ on $\Omega$. As in the case of homotopy morphisms, it is easy to check that 
\begin{itemize}
\item[-] this rule defines a homotopy cooperad, that is, once these maps are used to define a derivation of the free operad $\mathcal{F}(\mathcal{W})$, the resulting derivation squares to zero,
\item[-] the structure of an algebra over the cobar complex $\Omega(\mathcal{V}_{\bullet\to\bullet,\Omega})$ is equivalent to the datum of a pair of two homotopy $\mathcal{P}$-algebras $X$ and $Y$ and a homotopy morphism between $X$ and $Y\otimes\Omega$.  
\end{itemize}
We call the cobar complex $\Omega(\mathcal{V}_{\bullet\to\bullet,\Omega})$ the concordance operad, and denote it by $\mathcal{P}_{\bullet\to\bullet,\Omega}$.

\begin{theorem}\label{Th:concordance-operad}
The operad $\mathcal{P}_{\bullet\to\bullet,\Omega}$ is a resolution of the operad $\mathcal{P}_{\bullet\to\bullet}$.
\end{theorem}

\begin{proof}
The inclusion $\k\alpha_0\to\Omega^\vee$ is a quasi-isomorphism of dg cocommutative coalgebras. So the induced map
 \[
\mathcal{V}_{\bullet\to\bullet,\infty}\cong \overline{\mathcal{P}^\ac}_{(x\to x)}\oplus \overline{\mathcal{P}^\ac}_{(y\to y)}\oplus s\mathcal{P}^\ac_{(x\to y)}\otimes\k\alpha_0\to \mathcal{V}_{\bullet\to\bullet,\Omega}
 \]
is a morphism of homotopy cooperads. Let us show that the arising morphism of cobar complexes
 \[
\mathcal{P}_{\bullet\to\bullet,\infty} \to \mathcal{P}_{\bullet\to\bullet,\Omega}
 \]
is a quasi-isomorphism. That would be sufficient for our purposes: the first of these cobar complexes is the minimal model of the operad $\mathcal{P}_{\bullet\to\bullet}$, so we deduce that $\mathcal{P}_{\bullet\to\bullet,\Omega}$ is also a model of $\mathcal{P}_{\bullet\to\bullet}$, the quasi-isomorphism $\mathcal{P}_{\bullet\to\bullet,\Omega}\to \mathcal{P}_{\bullet\to\bullet}$ being given by the canonical projection onto the homology (there is a canonical projection since the quasi-isomorphism statement ensures that the homology is concentrated in degree zero	).

Let us consider the weight filtration of cobar complexes, that is the filtration by the number of internal vertices of trees in free operads. For the first cobar complex, the differential $d_0$ of the associated spectral sequence is zero, for the second one, it is given by the dual of the de Rham differential without any decompositions coming from $\mathcal{P}^\ac$. Therefore at the page $E_1$ of the corresponding spectral sequence, the map induced by the morphism of cobar complexes is an isomorphism. Since the weight filtration is exhaustive and bounded below, it follows from the mapping theorem for spectral sequences~\cite{MacLane95} that the two cobar complexes have the same homology. \qed
\end{proof}

\subsection{Convergent homotopy retracts: a toy example}

Na\"ively, it is natural to assume that the \v{C}ech chain complex $C_\bullet([0,1])$ of the interval should play a crucial role in defining homotopies algebraically. However, our constructions rather use the de Rham complex, or its linear dual. To repair this discrepancy, we shall contract the graded dual of the de Rham complex $\Omega^\vee$ onto its subcomplex isomorphic to $C_\bullet([0,1])$. The cost of that is that some higher structures emerge. In \cite{CG}, Cheng and Getzler performed the dual computation, exhibiting a homotopy commutative algebra structure on the cochain complex. In \cite{BM1,BM2}, Buijs and Murillo compute the transferred coalgebra, working however with the full linear dual, not the graded one. As a consequence, they have to invoke computations with formulae for which the end result makes perfect sense, but intermediate computations go outside the universe of coalgebras, since they involve completions of various coalgebras involved (those formulae are very close to duals of those in~\cite{CG}). We shall aim to make their computations completely rigorous, obtaining a ``converging'' sequence of homotopy retracts that produce a hierarchy of transferred structures on various constructions involving the chain complex of the interval, which ``stabilize at $\infty$'' and produce various meaningful higher structures. Let us be more precise about it. 

Let us denote by $\Omega^\vee_{(N)}$ the linear span of the elements $\alpha_{j}$ with $j\ge N$ and $\beta_{j}$ with $j\ge N-1$. The subspaces $\Omega^\vee_{(N)}$ form a decreasing filtration of $\Omega$ with $\bigcap_N\Omega^\vee_{(N)}=0$, and are compatible with the coproduct in the sense that 
 \[
\delta(\Omega^\vee_{(N)})\subset \sum_{k=0}^N \Omega^\vee_{(k)}\otimes\Omega^\vee_{(N-k)}\subset\Omega^\vee\otimes\Omega^\vee.
 \]

\begin{proposition}\label{retract}
There exists a sequence of homotopy retracts 
 \[
\xymatrix{     *{ \quad \ \  \quad (\Omega^\vee, d)\ } \ar@(dl,ul)[]^{K_N}\ \ar@<0.5ex>[r]^-{\theta} & *{\
(C_\bullet([0,1]),d)\  \ \quad }  \ar@<0.5ex>[l]^-{\omega_N}}
 \]
with $\id_{\Omega^\vee}-\omega_N\theta=dK_N+K_Nd$, for which the map $\theta$ vanishes on $\Omega^\vee_{(2)}$, and for all $s\ge 0$ the images of the maps $\omega_N-\omega_{N+s}$ and $K_N-K_{N+s}$ are contained in $\Omega^\vee_{(N)}$.
\end{proposition}

\begin{proof}
We define the map $\theta$ so that it identifies the subcomplex spanned by $1^\vee$, $t^\vee$, $dt^\vee$ with $(C_\bullet([0,1]),d)$ as follows:
\begin{gather}\label{projection}
\theta(\alpha_i)=
\begin{cases}
\mathbf{0}, \qquad i=0,\\
\mathbf{1}-\mathbf{0}, \quad i=1,\\
0, \qquad i>1,
\end{cases}
\theta(\beta_i)=
\begin{cases}
\mathbf{01}, \quad i=0,\\
0, \quad i>0.
\end{cases}
\end{gather}
We also put 
\begin{equation}\label{embedding}
\omega_N(\mathbf{0})=\alpha_0,\quad \omega_N(\mathbf{1})=\sum_{p=0}^{N-1}\alpha_p,\quad \omega_N(\mathbf{01})=\sum_{p=0}^{N-2}\frac{\beta_p}{p+1},
\end{equation}
and
\begin{gather}\label{contraction}
 K_N(\alpha_i)=
\begin{cases}
0, \qquad i=0,\\
-\sum_{j=1}^{N-2}\frac{\beta_j}{j+1}, \quad i=1,\\
\frac{\beta_{i-1}}{i},\quad i>1,
\end{cases}
K_N(\beta_i)=0.
\end{gather}
We have
 \[
(dK_N+K_Nd)(\alpha_i)=dK_N(\alpha_i)=
\begin{cases}
0, \qquad i=0,\\
-\sum_{j=2}^{N-1}\alpha_j, \quad i=1,\\
\alpha_i,\quad i>1,
\end{cases}
 \]
and
 \[
(dK_N+K_Nd)(\beta_i)=K_Nd(\beta_i)=K_N((i+1)\alpha_{i+1})=
\begin{cases}
-\sum_{j=1}^{N-2}\frac{\beta_j}{j+1}, \quad i=0,\\
\beta_i, \qquad i>0.
\end{cases}
 \]
Comparing these with Formulae \eqref{embedding} and \eqref{projection} above, we immediately conclude that 
 \[
dK_N+K_Nd=\id_{\Omega^\vee}-\omega_N\theta,   
 \]
as required.
Also,
 \[
(\omega_N-\omega_{N+s})(\mathbf{1})=-\sum_{p=N}^{N+s-1}\alpha_p, \quad (\omega_N-\omega_{N+s})(\mathbf{01})=-\sum_{p=N-1}^{N+s-2}\frac{\beta_p}{p+1},
 \]
and
\begin{gather*}
(K_N-K_{N+s})(\alpha_i)=
\begin{cases}
0, \qquad i=0,\\
\sum_{j=N-1}^{N+s-2}\frac{\beta_j}{j+1}, \quad i=1,\\
0,\quad i>1,
\end{cases}
\end{gather*}
and $\omega_N(\mathbf{0})$, $K_N(\beta_i)$ do not depend on $N$ at all, which proves the second claim.\qed
\end{proof}

The main consequence of the result that we just proved is that we can use it to obtain higher structures that are out of reach otherwise, making sense of computations with infinite sums that only exist in the completion of $\Omega^\vee$ with respect to filtration by subspaces~$\Omega^\vee_{(N)}$.  As a toy model, let us recall how one can recover (the universal enveloping algebra of) the Lawrence-Sullivan dg Lie algebra $\mathfrak{L}_{LS}$ using this retract. 

\begin{proposition}\label{toy-example}
The $A_\infty$-coalgebra structure of the algebra $A_{LS}$ (whose underlying chain complex $\langle u,v,w\rangle$ is isomorphic to $C_{\bullet}([0,1])$) is precisely the limit of structures obtained from the dg coalgebra structure on $\Omega^\vee$ by homotopy transfer formulae along the homotopy retracts from Proposition~\ref{retract}.
\end{proposition}

\begin{proof}
A computation that uses homotopy transfer involving the completion of $\Omega^\vee$ is presented in~\cite{BM1,BM2}, however, that proof has to invoke, at intermediate stages, infinite sums which are not well defined (mainly because the contracting homotopy $K=\lim_{N\to\infty}K_N$ does not restrict to the subspace of the completion of $\Omega^\vee$ with which the authors choose to work). Let us outline a way to fix that problem. We shall show that for our sequence of retracts the transferred map $\delta_{k,N}$ obtained from the $N^\text{th}$ retract does not depend on $N$ for $k<N$. Indeed, if we replace $\omega_N$ by $\omega_{N+s}$ and $K_N$ by $K_{N+s}$, this would change results of intermediate computations by elements from $\Omega^\vee_{(N)}$. Iterations of less than $N$ decompositions of those would produce a tensor product where at least one factor belongs to $\Omega^\vee_{(2)}$, hence will be annihilated at the final step when we apply the map $\theta$ everywhere (this map vanishes on $\Omega^\vee_{(2)}$ by construction). This guarantees that the computation in the spirit \cite{BM1,BM2}, even though either transfers higher structures from a space which is not a coalgebra or uses the contracting homotopy with the wrong codomain, nevertheless produces an honest $A_\infty$-coalgebra. \qed
\end{proof}

\begin{remark}
Of course, in this particular toy example, one could simply dualise the statement of Cheng and Getzler~\cite{CG} to obtain the same result. However, for our main computation, the strategy outlined seems a much more straightforward way to proceed, hence the toy example introducing this strategy.
\end{remark}

\subsection{Higher structures arising from the concordance operad}

We now are ready to relate the operadic formulation of concordance to Markl's approach to homotopy. The strategy for that is to use once again homotopy transfer along a convergent sequence of homotopy retracts. Let us consider the dg $\mathbb{S}$-module
 \[
\mathcal{V}_{\bullet\to\bullet,\Omega}=\overline{\mathcal{P}^\ac}_{(x\to x)}\oplus \overline{\mathcal{P}^\ac}_{(y\to y)}\oplus s\mathcal{P}^\ac_{(x\to y)}\otimes\Omega^\vee
 \]
which carries a homotopy cooperad structure. It is natural to try and transfer that homotopy cooperad structure to
 \[
\mathcal{V}_{\bullet\rightrightarrows\bullet,\infty}:=\overline{\mathcal{P}^{\ac}}_{(x\to x)}\oplus \overline{\mathcal{P}^{\ac}}_{(y\to y)}\oplus s\mathcal{P}^{\ac}_{(x\to y)}\otimes C_\bullet([0,1]) .
 \]
For that, let us mimic the set-up of Proposition~\ref{retract}, and consider the filtration of $\mathcal{V}_{\bullet\to\bullet,\Omega}$ by subspaces $(\mathcal{V}_{\bullet\to\bullet,\Omega})_{(N)}$, where $(\mathcal{V}_{\bullet\to\bullet,\Omega}){(0)}=\mathcal{V}_{\bullet\to\bullet,\Omega}$, and for $N>0$ 
 \[
(\mathcal{V}_{\bullet\to\bullet,\Omega})_{(N)}=s\mathcal{P}^\ac_{(x\to y)}\otimes\Omega^\vee_{(N)}.
 \]

\begin{theorem}\label{OperadicHomotopy}
Let $\mathcal{P}$ be a Koszul operad. The homotopy hypothesis holds for  $\mathcal{P}$.  More precisely,
\begin{enumerate}
\item There exists a sequence of homotopy retracts 
\[
\xymatrix{     *{ \qquad \ \ \  \quad (\mathcal{V}_{\bullet\to\bullet,\Omega}, d)\ } \ar@(dl,ul)[]^{H_N}\ \ar@<0.5ex>[r]^-{p} & *{\
		(\mathcal{V}_{\bullet\rightrightarrows\bullet,\infty},d)\  \ \quad }  \ar@<0.5ex>[l]^-{i_N}}
\]
with $\id_{\mathcal{V}_{\bullet\to\bullet,\Omega}}-i_Np=dH_N+H_Nd$, for which the map $p$ vanishes on $(\mathcal{V}_{\bullet\to\bullet,\Omega})_{(2)}$, and for each $s\ge 0$ the images of the maps $i_N-i_{N+s}$ and $H_N-H_{N+s}$ are contained in $(\mathcal{V}_{\bullet\to\bullet,\Omega})_{(N)}$.

\item The homotopy cooperad structure maps obtained by homotopy transfer along these homotopy retracts stabilise as $N\to\infty$. The limiting homotopy cooperad structure on $\mathcal{V}_{\bullet\rightrightarrows\bullet,\infty}$ has  the leading terms prescribed by the homotopy hypothesis. 
\end{enumerate}

\end{theorem}

\begin{proof}
Note that the differential of $\mathcal{V}_{\bullet\to\bullet,\Omega}$ comes precisely from the dual of the de Rham differential on $\Omega^\vee$. Thus, each homotopy retract
 \[
\xymatrix{     *{ \quad \ \  \quad (\Omega^\vee, d)\ } \ar@(dl,ul)[]^{K_N}\ \ar@<0.5ex>[r]^-{\theta} & *{\
(C_\bullet([0,1]),d)\  \ \quad }  \ar@<0.5ex>[l]^-{\omega_N}}
 \]
from Proposition~\ref{retract} gives rise to a homotopy retract
 \[
\xymatrix{     *{ \qquad \ \ \  \quad (\mathcal{V}_{\bullet\to\bullet,\Omega}, d)\ } \ar@(dl,ul)[]^{H_N}\ \ar@<0.5ex>[r]^-{i} & *{\
(\mathcal{V}_{\bullet\rightrightarrows\bullet,\infty},d)\  \ \quad }  \ar@<0.5ex>[l]^-{p_N}}
 \]
with
\begin{gather*}
H_N(v_1,v_2,sv_3\otimes\lambda)=(0,0,sv_3\otimes K_N(\lambda)),\\
i_N(v_1,v_2,sv_3\otimes c)=(v_1,v_2,sv_3\otimes\omega_N(c)),\\
p(v_1,v_2,sv_3\otimes\lambda)=(v_1,v_2,sv_3\otimes\theta(\lambda)).
\end{gather*}
Let us explain why the transferred structure maps converge to a limit as $N\to\infty$. First, let us note that the filtration of $\mathcal{V}_{\bullet\to\bullet,\Omega}$ by the subspaces $(\mathcal{V}_{\bullet\to\bullet,\Omega})_{(N)}$ is compatible with the homotopy cooperad structure in the following sense: for each decomposition map $\Delta_t$ of this cooperad structure, the result lands in the space of tree-shaped tensors 
 \[
\bigotimes_{v\text{ a vertex of }t}(\mathcal{V}_{\bullet\to\bullet,\Omega})_{(N_v)}  
 \]
with $\sum_v N_v=N$. This compatibility property implies that for our sequence of homotopy retracts the transferred map $\Delta_{t,N}$ obtained from the $N^\text{th}$ retract does not depend on $N$ if the number of internal vertices of $t$ is less than $N$. Indeed, if we replace $i_N$ by $i_{N+s}$ and $H_N$ by $H_{N+s}$, this would change results of intermediate computations by elements from $(\mathcal{V}_{\bullet\to\bullet,\Omega})_{(N)}$. Iterated decompositions of those that result in less than $N$ parts would produce a tree shaped tensor where at least one factor belongs to $(\mathcal{V}_{\bullet\to\bullet,\Omega})_{(2)}$, hence will be annihilated at the final step when we apply the map $p$ everywhere (this map vanishes on $(\mathcal{V}_{\bullet\to\bullet,\Omega})_{(2)}$ by construction).

It follows that to compute the transferred structure, we may apply Formulae~(\ref{HTTforCoops}), and perform all computations in the completion of $\mathcal{V}_{\bullet\to\bullet,\Omega}$ with respect to the filtration $(\mathcal{V}_{\bullet\to\bullet,\Omega})_{(N)}$, without worrying of any convergence issues. In what follows, we denote by $i$ and $H$ the limits of the corresponding maps; they now range in the completion of~$\mathcal{V}_{\bullet\to\bullet,\Omega}$; the meaning of the notation $\omega$ and $K$ is the same, with these maps ranging in the completion of $\Omega^\vee$.

We first note that since maps of the homotopy retract do not interact with the components $\overline{\mathcal{P}^\ac}_{(x\to x)}$ and $\oplus \overline{\mathcal{P}^\ac}_{(y\to y)}$ of $\mathcal{V}_{\bullet\to\bullet,\Omega}$, the transferred structure on these components coincides with the structure before transfer.  We also note that the elements $\omega(\mathbf{0})$ and $\omega(\mathbf{1})$ of $\Omega^\vee$ satisfy the conditions 
\begin{gather*}
\delta(\omega(\mathbf{0}))=\delta(\alpha_0)=\alpha_0\otimes\alpha_0=\omega(\mathbf{0})\otimes\omega(\mathbf{0}),\\
\delta(\omega(\mathbf{1}))=\delta\left(\sum_{i\ge0}\alpha_i\right)=\sum_{i\ge 0}\sum_{a+b=i}\alpha_a\otimes\alpha_b=\omega(\mathbf{1})\otimes\omega(\mathbf{1}),\\
K(\omega(\mathbf{0}))=K(\alpha_0)=0,\\
K(\omega(\mathbf{1}))=K\left(\sum_{i\ge0}\alpha_i\right)=
-\sum_{j\ge 1}\frac{\beta_j}{j+1}+\sum_{i>1}\frac{\beta_{i-1}}{i}=0.
\end{gather*}
This implies that for both $\mathcal{P}^\ac_{(x\to y)}\otimes\mathbf{0}$ and
$\mathcal{P}^\ac_{(x\to y)}\otimes\mathbf{1}$, the transferred homotopy cooperad structure also comes exactly from the sub-cooperads 
$\mathcal{P}^\ac_{(x\to y)}\otimes\omega(\mathbf{0})$ and
$\mathcal{P}^\ac_{(x\to y)}\otimes\omega(\mathbf{1})$ of $\mathcal{V}_{\bullet\to\bullet,\Omega}$. One concludes that this part of the homotopy cooperad structure matches that of $\mathcal{V}_{\bullet\to\bullet,\infty}$, since both $\mathcal{V}_{\bullet\to\bullet,\Omega}$ and $\mathcal{V}_{\bullet\to\bullet,\infty}$, by their very construction, encode morphisms of bar complexes. 

Note that the leading terms in Formula~\eqref{HomotopyLT} come from the cooperations indexed by the only two trees with $n$ leaves that only have vertices with $n$ inputs and vertices with one input, the trees
 \[
 \xygraph{
!{<0pt,0pt>;<10pt,0pt>:<0pt,-10pt>::}
!{(1,1.6)}*{\scriptstyle{}}="a"
!{(-0.6,-1.6)}*{\scriptstyle{}}="e"
!{(2.6,-1.6)}*{\scriptstyle{}}="f"
!{(1,1)}*{\circ}="d"
!{(1,0)}*{\circ}="b"
!{(1,-1.6)}*{\scriptstyle{\ldots}}
"a"-"d"
"d"-"b"
"b"-"e"
"b"-"f"
}
\quad
\text{and}
\quad
 \xygraph{
!{<0pt,0pt>;<10pt,0pt>:<0pt,-10pt>::}
!{(1,1.6)}*{\scriptstyle{}}="a"
!{(-1.6,-1.6)}*{\scriptstyle{}}="e"
!{(3.6,-1.6)}*{\scriptstyle{}}="f"
!{(1,1)}*{\circ}="b"
!{(-1,-1)}*{\circ}="c"
!{(1,-1)}*{\scriptstyle{\ldots}}
!{(3,-1)}*{\circ}="d"
"a"-"b"
"b"-"c"
"b"-"d"
"c"-"e"
"d"-"f"
} \qquad .
 \]
For each of these trees, the contribution of nontrivial expansions of trees with at least two vertices is zero. (It is obvious for the first tree, and for the second tree follows from the fact that such a nontrivial substitution $(((t_1\circ_{j_1} t_2)\circ_{j_2} t_3) \cdots )\circ_{j_k} t_{k+1}$ would have a tree with inputs of both colours as $t_1$, and each decomposition map $\Delta_{t_1}$ for such a tree $t_1$ vanishes on the cooperad $\mathcal{V}_{\bullet\to\bullet,\Omega}$). This means that all the homotopy transfer computations simplify drastically, and the corresponding transferred cooperad maps $\tilde{\Delta}_t$ are given by the na\"ive formula $\tilde{\Delta}_t=t(p)\circ\Delta_t\circ i$. Let us show how the leading terms of \eqref{HomotopyLT} appear in this computation.

We wish to investigate the transferred homotopy cooperad decompositions $\tilde{\Delta}_t$ evaluated on elements
 \[
M_{(x\to y)}\otimes\mathbf{01}\in\mathcal{P}^\ac_{(x\to y)}\otimes C_\bullet([0,1]).
 \]
We instantly recover the leading term
 \[
\mathbf{01}\circ_1 M_{(x\to x)}\in \left(\mathcal{P}^\ac_{(x\to y)}(x)\otimes C_\bullet([0,1])\right)\circ_{(1)} \overline{\mathcal{P}^\ac_{(x\to x)}}
\]
corresponding to the infinitesimal decomposition. However, for the leading term that lands in the space
\[\mathcal{P}^\ac_{(y\to y)}\circ \left(\mathcal{P}^\ac_{(x\to y)}(x)\otimes C_\bullet([0,1])\right), \] the computation is less obvious. The $C_\bullet([0,1])$-label of the corresponding leading term is precisely
 \[
(\theta^{\otimes n}\circ\delta^{n-1}\circ \omega)(\mathbf{01}).
 \]
Let us compute that decoration explicitly. We have
\begin{multline}\label{maincalc}
(\theta^{\otimes n}\circ\delta^{n-1}\circ\omega)(\mathbf{01})=
(\theta^{\otimes n}\circ\delta^{n-1})\left(\sum_{i\ge0}\frac{\beta_i}{i+1}\right)=\\=
\sum_{i\ge 0}\frac{1}{i+1}\theta^{\otimes n}\left(\sum_{i_1+\ldots+i_n=i}\sum_{j=1}^{n}\alpha_{i_1}\otimes\cdots\otimes\alpha_{i_{j-1}}\otimes\beta_{i_{j}}\otimes\alpha_{i_{j+1}}\otimes\cdots\otimes\alpha_{i_n}\right).
\end{multline}
Let us concentrate on the term $j=n$ in the third sum for the moment. Recalling the definition of $\theta$, we conclude that we must have $i_n=0$, and $i_k\in\{0,1\}$ for $k<n$. Together with the condition $i_1+\cdots+i_n=i$, this means that after applying $\theta$ we end up with a sum over all $i$-element subsets of $\{1,\ldots,n-1\}$, and the tensor product has $\mathbf{1}-\mathbf{0}$ on the places indexed by the given subset, and $\mathbf{0}$ otherwise. Since the total sum obviously lands in the subspace of tensors symmetric in the first $n-1$ factors, we may rewrite it as
\begin{multline*}
\sum_{i\ge 0}\frac{1}{i+1}\binom{n-1}{i}\mathbf{0}^{\odot n-1-i}\odot(\mathbf{1}-\mathbf{0})^{\odot i}\otimes\mathbf{01}=\\=
\sum_{i\ge 0}\frac{1}{n}\binom{n}{i+1}\mathbf{0}^{\odot n-1-i}\odot(\mathbf{1}-\mathbf{0})^{\odot i}\otimes\mathbf{01}=
\frac{1}{n}\sum_{k=0}^{n-1}\mathbf{0}^{\odot n-1-k}\odot\mathbf{1}^{\odot k}\otimes\mathbf{01}.
\end{multline*}
Here we used the formulae $\frac{1}{i+1}\binom{n-1}{i}=\frac{1}{n}\binom{n}{i+1}$ and
 \[
\sum_{i=0}^{n-1}\binom{n}{i+1}a^{n-1-i}b^i=\sum_{k=0}^{n-1}a^{n-1-k}(a+b)^k,
 \]
the latter valid in any commutative ring (and is proved in $\mathbb{Z}[a,b]$ by noticing that both the left hand side and the right hand side are equal to the same expression $\frac{(a+b)^n-a^n}{b}$).

Now we recall the contributions of all individual $j=1,\ldots, n$ from~\eqref{maincalc}, and notice that the factor~$\frac{1}{n}$ precisely contributes to creating from all these contributions the term
 \[
\sum_{j=0}^{n-1}\mathbf{0}^{\odot n-1-j}\odot\mathbf{01}\odot\mathbf{1}^{\odot j}.
 \]
This is exactly the same as the element
\[
[[h]]=\Sym\left(h\otimes q^{\otimes(n-1)}+p\otimes h\otimes q^{\otimes(n-2)}+\cdots+p^{\otimes(n-1)}\otimes h\right)
\]
appearing in Formula~\eqref{HomotopyLT}, which completes the proof.  \qed
\end{proof}

We denote by $\mathcal{P}_{\bullet\rightrightarrows\bullet,\infty}$ the cobar complex of the homotopy cooperad $\mathcal{V}_{\bullet\rightrightarrows\bullet,\infty}$ that we just computed in the proof of Proposition~\ref{OperadicHomotopy}. 

\begin{theorem}\label{Th:Markl-operad}
The operad $\mathcal{P}_{\bullet\rightrightarrows\bullet,\infty}$ is a resolution of the operad $\mathcal{P}_{\bullet\to\bullet}$.
\end{theorem}

\begin{proof}
Since the homotopy cooperad structure on $\mathcal{V}_{\bullet\rightrightarrows\bullet,\infty}$ is obtained from that on $\mathcal{V}_{\bullet\to\bullet,\Omega}$ by homotopy transfer, we conclude, using Theorems \ref{Th:concordance-operad} and \ref{OperadicHomotopy} together with the general results on homotopy transfer \cite[Th.~10.3.1]{LV} and existence of inverses for homotopy quasi-isomorphisms \cite[Th.~10.4.4]{LV}, applied to homotopy (co)operads as (co)algebras over an appropriate coloured Koszul operad, that there exist homotopy quasi-isomorphisms	
 \[	
\mathcal{V}_{\bullet\to\bullet,\infty}\stackrel{\sim}{\leadsto} \mathcal{V}_{\bullet\to\bullet,\Omega}\stackrel{\sim}{\leadsto}\mathcal{V}_{\bullet\rightrightarrows\bullet,\infty} .
 \]	
The arising morphism of cobar complexes
 \[
\mathcal{P}_{\bullet\to\bullet,\infty}\to \mathcal{P}_{\bullet\rightrightarrows\bullet,\infty} 
 \]
is a quasi-isomorphism for the same reason as in Theorem~\ref{Th:concordance-operad} (the weight filtration and the mapping theorem for spectral sequences~\cite{MacLane95}), and that completes the proof.\qed
\end{proof}

\begin{corollary}
The notion of operadic homotopy is homotopically equivalent to the notion of concordance. More precisely, we have the equivalence of homotopy categories of algebras
 \[
\mathsf{Ho}(\mathcal{P}_{\bullet\rightrightarrows\bullet,\infty}\text{-}\mathsf{alg})\cong\mathsf{Ho}(\mathcal{P}_{\bullet\to\bullet,\Omega}\text{-}\mathsf{alg}).
 \] 
\end{corollary}

\begin{proof}
By Theorems \ref{Th:concordance-operad} and \ref{Th:Markl-operad}, we know that the operads $\mathcal{P}_{\bullet\to\bullet,\Omega}$ and $\mathcal{P}_{\bullet\rightrightarrows\bullet,\infty}$ are both resolutions of $\mathcal{P}_{\bullet\to\bullet}$. Since we work in characteristic zero, all operads are split, and all quasi-isomorphisms of operads are compatible with splittings. Therefore, by \cite[Th.~4.7.4]{Hinich97}, we see that 
 \[
\mathsf{Ho}(\mathcal{P}_{\bullet\to\bullet,\Omega}\text{-}\mathsf{alg})\cong\mathsf{Ho}(\mathcal{P}_{\bullet\to\bullet}\text{-}\mathsf{alg})\cong\mathsf{Ho}(\mathcal{P}_{\bullet\rightrightarrows\bullet,\infty}\text{-}\mathsf{alg}).  
 \]
\qed
\end{proof}

\section{Further directions}\label{future}

One possible direction where our homotopy transfer approach might be useful for ``de-mystifying'' the story is a conjecture made in the end of~\cite{Mar2}. That conjecture suggests, for every operad $\mathcal{P}$ admitting a minimal model $(\mathcal{F}(\mathcal{V}),d)$ and every small category $\mathsf{C}$ with a chosen cofibrant replacement $(\mathsf{F}(\mathsf{W}),\partial)$ of $\k\mathsf{C}$, the existence of a cofibrant replacement
 \[
(\mathcal{F}(\mathcal{V}\otimes\k\mathop{\mathrm{Ob}}(\mathsf{C})\oplus\mathsf{W}\oplus s\mathcal{V}\otimes \mathsf{W}),\mathbf{d})
 \]
for any coloured operad $\mathcal{O}_{\mathcal{P},\mathfrak{D}}$ describing $\mathcal{P}$-algebras and morphisms between them that form a diagram of shape $\mathfrak{D}$. The differential $\mathbf{d}$ of this replacement is conjectured to have a specific shape~\cite{Mar2}. A special case of this conjecture is proved in the case of a Koszul operad $\mathcal{P}$ with all generators of the same arity and degree in~\cite{Doubek}. We hope that homotopy transfer techniques might be the right tool to prove this conjecture in full generality in the Koszul case.

Another natural question to address in future work is to apply homotopy transfer theorems for homotopy retracts from de Rham complexes to \v{C}ech complexes beyond the case of the interval. It would be interesting already in the case of contractible spaces, for example for higher-dimensional simplexes and higher-dimensional disks the corresponding computation would contain further information on the higher dimensional categorification of algebras.

Further, while we concentrated on the case of a Koszul operad $\mathcal{P}$, it would be interesting to generalise the relevant notions and result to the case of any operad admitting a minimal model $(\mathcal{F}(\mathcal{V}),d)$, putting $\mathcal{P}^\ac:=s\mathcal{V}$, and making necessary adjustments in the view of the fact that $\mathcal{P}^\ac$ is no longer an honest cooperad but rather a homotopy cooperad. 

Finally, using the results of the present paper, we are investigating, in works in progress, homotopies of homotopy morphisms of homotopy Loday algebras \cite{AP10}, homotopies of morphisms of Lie $n$-algebroids \cite{BP12} and of Loday algebroids \cite{GKP11}.

\bibliographystyle{amsplain}

\providecommand{\bysame}{\leavevmode\hbox to3em{\hrulefill}\thinspace}
\providecommand{\MR}{\relax\ifhmode\unskip\space\fi MR }
\providecommand{\MRhref}[2]{%
  \href{http://www.ams.org/mathscinet-getitem?mr=#1}{#2}
}
\providecommand{\href}[2]{#2}

\end{document}